\documentclass{amsart}

\usepackage{amssymb}
\usepackage[hidelinks]{hyperref}
\usepackage[textsize=footnotesize]{todonotes}
\usepackage{tikz}
\usetikzlibrary{cd}
\usepackage{float}

\newtheorem{theorem}{Theorem}[section]
\newtheorem{corollary}[theorem]{Corollary}
\newtheorem{lemma}[theorem]{Lemma}
\newtheorem{proposition}[theorem]{Proposition}
\newtheorem{problem}[theorem]{Problem}
\theoremstyle{definition}
\newtheorem{definition}[theorem]{Definition}
\newtheorem{remark}[theorem]{Remark}
\newtheorem{example}[theorem]{Example}

\newcommand{\cA}{\mathcal{A}}

\newcommand{\cD}{\mathcal{D}}

\newcommand{\cF}{\mathcal{F}}

\newcommand{\cI}{\mathcal{I}}
\newcommand{\I}{\cI}
\newcommand{\cJ}{\mathcal{J}}
\newcommand{\J}{\cJ}

\newcommand{\cP}{\mathcal{P}}

\newcommand{\cU}{\mathcal{U}}
\newcommand{\cW}{\mathcal{W}}
\newcommand{\cZ}{\mathcal{Z}}
\newcommand{\continuum}{\mathfrak{c}}
\newcommand{\cc}{\continuum}
\newcommand{\fin}{\mathrm{Fin}}

\newcommand{\ED}{\mathcal{ED}} 
\newcommand{\EU}{\mathcal{EU}}
\newcommand{\conv}{\mathrm{conv}}

\newcommand{\BI}{\mathcal{BI}}

\begin{document} 


\title{Unboring ideals}


\author[A.~Kwela]{Adam Kwela}
\address[Adam Kwela]{Institute of Mathematics\\ Faculty of Mathematics\\ Physics and Informatics\\ University of Gda\'{n}sk\\ ul.~Wita  Stwosza 57\\ 80-308 Gda\'{n}sk\\ Poland}
\email{Adam.Kwela@ug.edu.pl}
\urladdr{http://kwela.strony.ug.edu.pl/}


\date{}






\begin{abstract}
Our main object of interest is the following notion: we say that a topological space space $X$ is in FinBW($\mathcal{I}$), where $\mathcal{I}$ is an ideal on $\omega$, if for each sequence $(x_n)_{n\in\omega}$ in $X$ one can find an $A\notin\I$ such that $(x_n)_{n\in A}$ converges in $X$. 

We define an ideal $\mathcal{BI}$ which is critical for FinBW($\mathcal{I}$) in the following sense: Under CH, for every ideal $\mathcal{I}$, $\mathcal{BI}\not\leq_K\mathcal{I}$ ($\leq_K$ denotes the Kat\v{e}tov preorder of ideals) iff there is an uncountable separable space in FinBW($\mathcal{I}$). We show that $\mathcal{BI}\not\leq_K\mathcal{I}$ and $\omega_1$ with the order topology is in FinBW($\mathcal{I}$), for all $\bf{\Pi^0_4}$ ideals $\mathcal{I}$. 

We examine when FinBW($\mathcal{I}$)$\setminus$FinBW($\mathcal{J}$) is nonempty: we prove under MA($\sigma$-centered) that for $\bf{\Pi^0_4}$ ideals $\mathcal{I}$ and $\mathcal{J}$ this is equivalent to $\mathcal{J}\not\leq_K\mathcal{I}$. Moreover, answering in negative a question of M. Hru\v{s}\'ak and D. Meza-Alc\'antara, we show that the ideal $\text{Fin}\times\text{Fin}$ is not critical among Borel ideals for extendability to a $\bf{\Pi^0_3}$ ideal. Finally, we apply our results in studies of Hindman spaces and in the context of analytic P-ideals.
\end{abstract}


\maketitle


\section{Introduction}

A nonempty collection $\mathcal{I}$ of subsets of a set $X$ is called an ideal on $X$ if it is closed under subsets and finite unions of its elements. In this paper we also assume throughout that $X\notin \I$ and $X=\bigcup \I$. All ideals considered in this paper are defined on infinite countable sets. By $\fin$ we denote the ideal of all finite subsets of $\omega=\{0,1,\ldots\}$. 

We treat the power set $\mathcal{P}(X)$ as the space $2^X$ of all functions $f:X\rightarrow 2$ (equipped with the product topology, where each space $2=\left\{0,1\right\}$ carries the discrete topology) by identifying subsets of $X$ with their characteristic functions. Thus, we can talk about descriptive complexity of subsets of $\mathcal{P}(X)$ (in particular, of ideals on $X$). 

The classical Bolzano-Weierstrass theorem states that $[0,1]$ is sequentially compact, that is, each sequence in it has a convergent subsequence. It is known that in general there are compact topological spaces that are not sequentially compact. On the other hand, $\omega_1$ with the order topology is a sequentially compact space which is non-compact.

The following modification of the notion of sequentially compact space gives us control on the size of the convergent subsequence.

\begin{definition}
\label{DefFinBW}
Let $\I$ be an ideal on a countable set $M$. By FinBW($\I$) we denote the class of all Hausdorff spaces $X$ such that for every sequence $(x_n)_{n\in M}\subseteq X$ there exists a set $A\notin\I$ such that $(x_n)_{n\in A}$ converges in $X$. If $X$ is in this class we briefly write $X\in\text{FinBW}(\I)$.
\end{definition}

Observe that FinBW($\fin$) is the class of all sequentially compact spaces. Note that the notation FinBW($\I$) is close to the one introduced in \cite{Gdansk}, where $\I\in\text{FinBW}$ means that $[0,1]\in\text{FinBW}(\I)$ in the sense of Definition \ref{DefFinBW}.

This subject has already been studied in the literature. There are at least two directions of those studies. The first one, related to the notions of Hindman space and Mr\'{o}wka space, was initiated by M. Kojman. In \cite{KojmanvdW} he considered spaces belonging to FinBW($\mathcal{W}$), where $\mathcal{W}$ denotes the van der Waerden ideal, that is:
$$\mathcal{W}=\left\{A\subseteq\omega:\ \left(\exists k\in\omega\right)\ A\text{ does not contain arithmetic progressions of length }k\right\}.$$
And in \cite[Theorem 3]{KojmanHindman} he showed that FinBW($\mathcal{H}$) is the class of all finite spaces, where $\mathcal{H}$ is the Hindman ideal. He also introduced  there a new class of Hindman spaces, to which we'll return in the last section of this paper. In \cite[Theorem 3]{KojmanShelah} M. Kojman and S. Shelah constructed under the continuum hypothesis (CH) an uncountable separable compact space from FinBW($\mathcal{W}$), which is not a Hindman space. This result has been extended by A. L. Jones to the case of Martin's axiom for $\sigma$-centered notions of forcing (MA($\sigma$-centered)) instead of CH. In both those results, the constructed space is a Mr\'{o}wka space (for the definition of Mr\'{o}wka space see subsection \ref{secMrowka}). Later J. Fla\v{s}kov\'{a} in \cite{Flaskova} examined FinBW($\mathcal{I}$) for $\bf{\Sigma^0_2}$ ideals $\I$ (in her paper spaces from FinBW($\mathcal{I}$) are called $\mathcal{I}$-spaces). She showed under MA($\sigma$-centered) that for every $\bf{\Sigma^0_2}$ ideal there is an uncountable separable compact Mr\'{o}wka space in FinBW($\mathcal{I}$) which is not Hindman. As far as we know, there are no results concerning the necessity of the assumptions CH and MA($\sigma$-centered) in the above studies.

The second direction concerns the following question: for which ideals $\mathcal{I}$ do we have $[0,1]\in\text{FinBW}(\mathcal{I})$? These studies were initiated by R. Filip\'{o}w, N. Mro\.{z}ek, I. Rec\l{}aw and P. Szuca  in \cite{Gdansk} and then continued in \cite{I-selection}, where it is proved, among other things, that if $\I$ is an $\bf{\Sigma^0_2}$ ideal then $[0,1]\in\text{FinBW}(\I|A)$ for all $A\notin\I$ \cite[Proposition 3.4]{Gdansk}. Finally, D. Meza-Alc\'antara in \cite[Section 2.7]{Meza} proved that the ideal $\conv$ on $[0,1]\cap\mathbb{Q}$ generated by all convergent sequences (this ideal is $\bf{\Sigma^0_4}$) is critical for FinBW: $[0,1]\in\text{FinBW}(\I)$ is equivalent to $\conv\not\leq_K\I$, where $\leq_K$ denotes the Kat\v{e}tov preorder on ideals (i.e., $\I\leq_K\J$, if there is a $f:\bigcup\J\to \bigcup\I$ such that $f^{-1}[A]\in\J$ for each $A\in\I$).

This article is organized as follows. The most interesting results are presented in Sections 6 and 10. Section 2 contains some basic definitions needed throughout the paper, that were not mentioned in this Introduction.  
The longest and most technical Section 4 as well as Sections 3 and 5 are devoted to some preliminary discussions. Section 6 contains the main results of this part of the paper. In particular, we prove that $\omega_1$ with the order topology is in FinBW($\mathcal{I}$), for all $\bf{\Pi^0_4}$ ideals $\mathcal{I}$, and find an ideal $\mathcal{BI}$ which serves as the boundary point for FinBW($\mathcal{I}$) containing interesting spaces: Under CH, for every ideal $\mathcal{I}$, $\mathcal{BI}\not\leq_K\mathcal{I}$ if and only if there is an uncountable separable space in FinBW($\mathcal{I}$). In Section 7 we pose some open problems and answer in negative a question of M. Hru\v{s}\'ak and D. Meza-Alc\'antara about extendability to a $\bf{\Pi^0_3}$ ideal. Sections 8 and 9 are devoted to preliminary results concerning the following question: when FinBW($\mathcal{I}$)$\setminus$FinBW($\mathcal{J}$) is nonempty? In Section 10 we prove (under MA($\sigma$-centered)) that for $\bf{\Pi^0_4}$ ideals $\mathcal{I}$ and $\mathcal{J}$ this is equivalent to $\mathcal{J}\not\leq_K\mathcal{I}$. Finally, in our last Section 11 we apply our results in studies of Hindman spaces and in the context of analytic P-ideals.

\section{Preliminaries}

In this section we collect some basic definitions and observations needed throughout the paper, that were not mentioned so far.

\subsection{Basic notions}

By a partition of a set $X$ we mean any family $(X_n)\subseteq\cP(X)$ such that $\bigcup_{n}X_n=X$ and $X_n\cap X_k=\emptyset$ whenever $n\neq k$. In particular, we do not require that the sets $X_n$ are infinite or non-empty.

Let $A,B\in[\omega]^\omega$. We say that $A$ and $B$ are almost disjoint if $A\cap B$ is finite. A family $\mathcal{A}\subseteq[\omega]^\omega$ is called an AD family if the members of $\mathcal{A}$ are pairwise almost disjoint. If moreover $\mathcal{A}$ is a maximal AD family with respect to inclusion, it is called a MAD family. Equivalently, $\mathcal{A}$ is a MAD family if it is an AD family and for each infinite $B\subseteq\omega$ there is an $A\in\mathcal{A}$ with $A\cap B$ infinite.

The restriction of an ideal $\I$ to $Y$ is given by $\I|Y=\{A\cap Y:\ A\in\I\}$. Note that $\I|Y$ is an ideal on $Y$ if and only if $Y\notin\I$. We say that an ideal $\I$ on $X$ is generated by the family $\mathcal{F}\subseteq\mathcal{P}(X)$ if 
$$\I=\left\{A\subseteq X:\ \left(\exists n\in\omega\right)\ \left(\exists F_0,\ldots,F_n\in\mathcal{F}\right)\ A\subseteq F_0\cup\ldots\cup F_n\right\}.$$

\begin{definition}
If $\I$ is an ideal then $\mathcal{D}_{\I}$ is the set of all functions $f:\bigcup\I\to\omega$ such that $f^{-1}[\{n\}]\in\I$ for all $n\in\omega$. In particular, $\mathcal{D}_{\fin}$ is the set of all finite-to-one functions from $\omega$ to $\omega$.
\end{definition}

If $\I$ and $\J$ are ideals on $X$ and $Y$, respectively, then:
\begin{itemize}
\item the product of $\I$ and $\J$ is an ideal on $X\times Y$ given by:
$$\I\otimes\J=\{A\subseteq X\times Y:\ \{x\in X:\ A_{(x)}\notin\J\}\in\I\},$$
where $A_{(x)}=\{y\in Y:\ (x,y)\in A\}$. Also with either $\I$ or $\J$ (but not both) replaced by $\{\emptyset\}$, the above formula defines ideals, denoted $\emptyset\otimes\J$ and $\I\otimes\emptyset$, respectively.
\item the disjoint sum of $\I$ and $\J$ is an ideal on $(\{0\}\times X)\cup(\{1\}\times Y)$ given by:
$$\I\oplus\J=\{A\subseteq (\{0\}\times X)\cup(\{1\}\times Y):\ A_{(0)}\in\I\text{ and }A_{(1)}\in\J\}.$$
\end{itemize}

If $X$ is an infinite countable set then by $\fin(X)$ we will denote the ideal of all finite subsets of $X$. Hence, $\fin=\fin(\omega)$. For an ideal $\I$ and $A\subseteq\bigcup\I$ we say that $A$ is an $\I$-positive set whenever $A\notin\I$. An ideal $\I$ on $X$ is: 
\begin{itemize}
\item tall if for each infinite $A\subseteq X$ there is a $B\subseteq A$ with $B\in\I\setminus\fin(A)$; 
\item a P-ideal if for each $(A_n)\subseteq\I$ there is an $A\in\I$ with $A_n\setminus A$ finite for all $n\in\omega$; 
\item maximal if for any $A\subseteq X$ either $A\in\I$ or $X\setminus A\in\I$ (equivalently, $\I$ is maximal with respect to inclusion). It is known that a maximal ideal cannot be Borel. 
\end{itemize}
We say that $\I$ is extendable to $\J$ if $\I\subseteq\J$. Thus, for instance, the statement "$\I$ is extendable to a P-ideal" means that there is a P-ideal $\J$ with $\I\subseteq\J$. In particular, every ideal is extendable to a maximal ideal.
 
For $\mathcal{A}\subseteq\mathcal{P}(X)$ we write
$$\mathcal{A}^\star=\{X\setminus A:\ A\in\mathcal{A}\}.$$
If $\I$ is an ideal on $X$ then $\I^\star$ is a filter (i.e., a family closed under supersets and finite intersections of its elements). We call it the dual filter of $\I$. Observe that the map $A\mapsto X\setminus A$ is continuous. Thus, $\I$ and $\I^\star$ have the same descriptive complexity. 

Dual filter of a maximal ideal is called an ultrafilter. Recall that an ultrafilter $\mathcal{U}\subseteq\cP(\omega)$ is a P-point if $\cU^\star$ is a P-ideal. It is known that P-points exist under CH as well as under MA($\sigma$-centered).

\subsection{Preorders on ideals}

Let $\I$ and $\J$ be ideals on $X$ and $Y$, respectively (here we make an exception and allow $\I=\mathcal{P}(X)$ or $\J=\mathcal{P}(Y)$ -- this will simplify some of our proofs). We say that:
\begin{itemize} 
\item $\J$ and $\I$ are isomorphic and write $\I\cong\J$, if there is a bijection $f:Y\to X$ such that 
$$\left(\forall A\subseteq X\right)\ f^{-1}[A]\in\J\ \Longleftrightarrow\ A\in\I;$$
\item $\J$ contains an isomorphic copy of $\I$ and write $\I\sqsubseteq\J$, if there is a bijection $f:Y\to X$ such that $f^{-1}[A]\in\J$ for each $A\in\I$;
\item $\J$ is above $\I$ in the Kat\v{e}tov-Blass preorder and write $\I\leq_{KB}\J$, if there is a finite-to-one $f:Y\to X$ (i.e., $f^{-1}[\{x\}]\in\fin(Y)$ for all $x\in X$) such that $f^{-1}[A]\in\J$ for each $A\in\I$;
\item $\J$ is above $\I$ in the Kat\v{e}tov preorder and write $\I\leq_K\J$, if there is an $f:Y\to X$ (not necessary finite-to-one) such that $f^{-1}[A]\in\J$ for each $A\in\I$;
\end{itemize}
Obviously, 
$$\I\cong\J\implies\I\sqsubseteq\J\implies\I\leq_{KB}\J\implies\I\leq_{K}\J.$$
The preorders $\leq_K$ and $\sqsubseteq$ were extensively studied e.g. in \cite{Kat}, \cite{Hrusak}, \cite{Hrusak2}, \cite{WR}, \cite{EUvsSD} and
\cite{Meza}. A property of ideals can often be expressed by finding a critical ideal (in sense of some preorder on ideals) with respect to this property. This approach proved to be especially effective in the context of FinBW ideals.

Note that $\I\leq_{KB}\I|A$ (as witnessed by the identity function) for all ideals $\I$ and all $A\notin\I$.

\begin{definition}[{\cite{HrusakKatetov} or \cite[Subsection 2.1.2]{Meza}}]
An ideal $\I$ is:
\begin{itemize}
\item $\leq_{K}$-uniform if $\I|A\leq_{K}\I$ for all $A\notin\I$; 
\item $\leq_{KB}$-uniform if $\I|A\leq_{KB}\I$ for all $A\notin\I$; 
\item homogeneous if $\I|A$ is isomorphic to $\I$ whenever $A\notin\I$.
\end{itemize}
\end{definition}

Obviously, a homogeneous ideal is $\leq_{KB}$-uniform and a $\leq_{KB}$-uniform ideal is $\leq_{K}$-uniform.

\begin{example}
\label{uniform-ex}
The following ideals are homogeneous:
\begin{itemize}
\item $\mathcal{W}$ (cf. \cite[Example 2.6]{Jacek});
\item $\mathcal{H}$ (cf. \cite[Example 2.6]{Jacek});
\item all maximal ideals (cf. \cite[Example 1.4]{Jacek});
\item $\fin^n$ for all $n\geq 1$, where $\fin^1=\fin$ and $\fin^{n+1}=\fin\otimes\fin^n$ (cf. \cite[Remark after Proposition 2.9]{Jacek}).
\end{itemize}
\end{example}

\begin{example}
\label{I_d is KB}
Define the density zero ideal 
$$\I_{d}=\left\{A\subseteq\omega:\ \lim_{n\to\infty}\frac{|A\cap\{0,1,\ldots,n\}|}{n+1}=0\right\}.$$ 
It is a $\bf{\Pi^0_3}$ P-ideal (cf. \cite[Example 1.2.3(d)]{Farah}). By \cite[Proposition 2.1]{HrusakKatetov}, it is also $\leq_{K}$-uniform (note that this fact is also stated in \cite[Proposition 2.1.11]{Meza}, but with an incorrect proof). Actually, the proof of \cite[Proposition 2.1]{HrusakKatetov} shows even more: $\I_{d}$ is $\leq_{KB}$-uniform.
\end{example}

\subsection{Related notions and their critical ideals}

Consider the ideal:
$$\fin^2=\fin\otimes\fin=\{A\subseteq\omega^2:\ \{n\in\omega:\ A_{(n)}\notin\fin\}\in\fin\}.$$
This ideal is $\bf{\Sigma^0_4}$. It is critical for the following property considered in the literature: We say that an ideal $\I$ on $X$ is a weak P-ideal, if for each partition $(A_n)$ of $X$ into sets belonging to $\I$ there is an $S\notin\I$ with $S\cap A_n$ finite, for all $n\in\omega$ (note that all P-ideals are weak P-ideals).

\begin{proposition}[{Essentially \cite{Reclaw} and \cite{Meza}}]
\label{o}
Let $\I$ be an ideal.
\begin{itemize}
\item[(a)] $[0,1]\in\text{FinBW}(\I)$ if and only if $\conv\not\sqsubseteq\I$;
\item[(b)] $\I$ is a weak P-ideal if and only if $\fin^2\not\sqsubseteq\I$.
\end{itemize}
\end{proposition}

\begin{proof}
Item (a) can be found is \cite[Section 2.7]{Meza}. For (b) see the proof of \cite[Lemma 2]{Reclaw}.
\end{proof}

In the whole paper we will use "$\fin^2\not\sqsubseteq\I$" and "$\I$ is a weak P-ideal" interchangeably.

\begin{proposition}[{\cite[Examples 4.1 and 4.4]{Kat}}]
\label{propertyKat}
Let $\I$ be any ideal and $\J=\conv$ or $\J=\fin^2$. The following are equivalent:
\begin{itemize}
\item $\mathcal{J}\leq_{K}\I$;
\item $\mathcal{J}\leq_{KB}\I$;
\item $\mathcal{J}\sqsubseteq\I$.
\end{itemize}
\end{proposition}

\begin{proposition}[{\cite[Theorem 2.8.7]{Meza} and Proposition \ref{propertyKat}}]
\label{P-point}
Let $\cU$ be an ultrafilter. The following are equivalent:
\begin{itemize}
\item $\cU$ is a P-point;
\item $\conv\not\leq_K\cU^\star$;
\item $\conv\not\sqsubseteq\cU^\star$;
\item $\fin^2\not\leq_K\cU^\star$;
\item $\fin^2\not\sqsubseteq\cU^\star$.
\end{itemize}
\end{proposition}

\subsection{Mr\'{o}wka spaces}
\label{secMrowka}

For an infinite AD family $\mathcal{A}$ of infinite subsets of $\omega$, define a topological space $\Phi(\mathcal{A})$ as follows: 
\begin{itemize}
\item the underlying set of $\Phi(\mathcal{A})$ is $\omega\cup\mathcal{A}\cup\{\infty\}$, 
\item the points of $\omega$ are isolated, 
\item each basic neighborhood of $A\in\mathcal{A}$ is of the form $\{A\}\cup
(A\setminus F)$, where $F\subseteq\omega$ is finite,
\item each basic neighborhood of $\infty$ is of the form 
$$\{\infty\}\cup(\mathcal{A}\setminus G)\cup\left(\omega\setminus\left(F\cup\bigcup_{A\in G}A\right)\right),$$
where both $F\subseteq\omega$ and $G\subseteq\mathcal{A}$ are finite.
\end{itemize}
Equivalently, $\Phi(\mathcal{A})$ is the one-point compactification of the space with $\omega\cup\mathcal{A}$ as the underlying set and basic neighborhoods of $A\in\mathcal{A}$ and $n\in\omega$ as above (this space was introduced in \cite{Mrowka}). Indeed, if $C$ is a compact subset of the latter space, then $C\cap\mathcal{A}$  and $(C\cap\omega)\setminus\bigcup_{A\in C\cap\mathcal{A}}A$ have to be finite. 

In this paper a space of the form $\Phi(\mathcal{A})$, for some AD family $\mathcal{A}$, will be called Mr\'{o}wka space.

Each Mr\'{o}wka space is Hausdorff, separable and compact. Moreover, it is first countable at every point of $\Phi(\mathcal{A})\setminus\{\infty\}$. If $\mathcal{A}$ is a MAD family, then $\Phi(\mathcal{A})$ is sequentially compact (see \cite[Section 2]{FT} or \cite[Theorem 6]{KojmanvdW}). Observe also that $\Phi(\mathcal{A})$ is countable if and only if $\mathcal{A}$ is countable. In particular, if $\mathcal{A}$ is a MAD family, then $\Phi(\mathcal{A})$ is uncountable (as MAD families are uncountable). For more on Mr\'{o}wka spaces see also \cite{HrusakMrowka} or \cite{Mrowka}.

\section{Boring spaces}

In this section we introduce a new class of topological spaces, which are in some sense the most boring examples of infinite sequentially compact spaces.

\begin{definition}
We say that a topological space $X$ is boring if $X$ is sequentially compact and there is a finite set $F\subseteq X$ such that each injective convergent sequence in $X$ converges to some point from $F$.
\end{definition}

\begin{proposition}\
\label{boringspaces}
\begin{itemize}
\item[(i)] For each cardinal $\kappa$, there is a Hausdorff boring space of cardinality $\kappa$. 
\item[(ii)] Let $X$ be an infinite Hausdorff space. Then $X$ is a boring space if and only if $X$ is a disjoint union of finitely many one-point compactifications of discrete spaces.
\item[(iii)] Each boring space is compact.
\item[(iv)] A boring space is separable if and only if it is countable.
\end{itemize}
\end{proposition}

\begin{proof}
(i): If $\kappa$ is finite, it suffices to take any discrete space of cardinality $\kappa$. If $\kappa$ is an infinite cardinal number, fix a set $Y$ of cardinality $\kappa$ and let $X$ be the one-point compactification of the space $(Y,\cP(Y))$. Then $X$ is as needed.

(ii): It is obvious that each disjoint union of finitely many one-point compactifications of discrete spaces is boring. Thus, we only need to prove the opposite implication.

Let $(X,\tau)$ be Hausdorff and boring. Let $\{x_0,\ldots,x_n\}\subseteq X$ be such that each injective convergent sequence in $X$ converges to some $x_i$, for $i\leq n$. Firstly, we will show that each $x\in X\setminus \{x_0,\ldots,x_n\}$ is isolated.

Fix any $x\in X\setminus \{x_0,\ldots,x_n\}$. Since $X$ is Hausdorff, there are pairwise disjoint open sets $V,V_0,\ldots,V_n$ such that $x\in V$ and $x_i\in V_i$, for $i\leq n$. Note that $V$ is finite, as otherwise it would contain an injective sequence which does not converge to any $x_i$. For each $y\in V\setminus\{x\}$ find open $U_y$ with $x\in U_y$ and $y\notin U_y$. Then $V\cap\bigcap_{y\in V\setminus\{x\}}U_y=\{x\}$ is open. Hence, $x$ is an isolated point.

Now we show that $X$ is a disjoint union of finitely many one-point compactifications of discrete spaces. Since $X$ is Hausdorff, there are pairwise disjoint open sets $U_0,\ldots,U_n$ such that $x_i\in U_i$, for $i\leq n$. Without loss of generality we may assume that $X=U_0\cup\ldots\cup U_n$. Indeed, the set $X\setminus (U_0\cup\ldots\cup U_n)$ is finite (otherwise it would contain an injective sequence with no convergent subsequence), so open (by the above paragraph) and we can add it to $U_0$. We claim that $U_i$ with the topology $\{V\cap U_i:\ V\in\tau\}$ is a one-point compactification of a discrete space, for each $i\leq n$. 

Fix $i\leq n$. We already know that each $x\in U_i\setminus\{x_i\}$ is isolated. Thus, we only need to show that for each open neighborhood $V$ of $x_i$ the set $U_i\setminus V$ is finite. This is true as otherwise $U_i\setminus V$ would contain an injective sequence with no convergent subsequence (as this sequence is outside some neighborhood of $x_j$, for each $j\leq n$), which contradicts sequential compactness of $X$.

(iii) and (iv): Straightforward (using item (ii)). 
\end{proof}

\begin{proposition}
\label{boringspaces2}
If $\fin^2\not\sqsubseteq\I$ then each boring space is in FinBW($\I$).
\end{proposition}

\begin{proof}
Fix a boring space $X$. By item (ii) of Proposition \ref{boringspaces}, $X$ is a disjoint union of one-point compactifications of discrete spaces. Thus, $X=\bigcup_{i\leq n}X_i$, where $X_i\cap X_j=\emptyset$ whenever $i\neq j$ and for each $i\leq n$ there is an $x_i$ which is a limit point of each injective convergent sequence in $X_i$. Fix any $f:\omega\to X$. If there is an $x\in X$ with $f^{-1}[\{x\}]\notin\I$, then $f|f^{-1}[\{x\}]$ converges to $x$ and we are done. On the other hand, if $f^{-1}[\{x\}]\in\I$ for all $x\in X$, then $(f^{-1}[\{x\}])_{x\in X}$ is a partition of $\omega$ into sets belonging to $\I$. As $\fin^2\not\sqsubseteq\I$, there is an $A\notin\I$ such that $f|A$ is finite-to-one. Since $A\notin\I$ and $A=\bigcup_{i\leq n}(A\cap f^{-1}[X_i])$, there is an $i\leq n$ with $B=A\cap f^{-1}[X_i]\notin\I$. Then $f|B$ converges to $x_i$ and we are done.
\end{proof}

\section{Boring ideals}

In this section we provide the combinatorial background needed in our considerations.

\subsection{The critical ideal}

\begin{definition}
We define the ideal $\BI=(\emptyset\otimes\fin^2)\cap(\fin\times\fin(\omega^2))$. Equivalently, 
$$\mathcal{BI}=\left\{A\subseteq\omega^3:\ \left(\exists k\right)\left(\left(\forall i<k\right)A_{(i)}\in\fin^2\ \wedge\ \left(\forall i\geq k\right)A_{(i)}\text{ is finite}\right)\right\}.$$
\end{definition}

\begin{remark}
\label{rem-Borel}
The ideal $\BI$ is $\bf{\Sigma^0_4}$, as
$$\mathcal{BI}=\left\{A\subseteq\omega^3:\ \left(\exists k\right)\left(\left(\left(\forall l<k\right)\left(\exists m\right)\left(\forall j>m\right)\left(\exists n\right)\left(\forall i>n\right)(l,j,i)\notin A\right) \right.\right.$$
$$\left.\left.\wedge\ \left(\left(\forall l\geq k\right)\left(\exists n,m\right)\left(\forall i,j\right)(j>m\ \vee\ i>n)\implies (l,j,i)\notin A\right)\right)\right\}.$$
Moreover, it is easy to see that $\BI$ is tall. Indeed, for each infinite $A\subseteq\omega^3$ either $A_{(i)}$ is infinite for some $i\in\omega$ or $A$ intersects infinitely many $\{i\}\times\omega^2$. In the former case use tallness of $\fin^2$ to get an infinite $B\subseteq A_{(i)}$ with $B\in\fin^2$ and conclude that $\{i\}\times B$ is an infnite subset of $A$ belonging to $\BI$. In the latter case find any selectora $S\subseteq\omega^3$ of the family $(A\cap(\{i\}\times\omega^2))$ and note that $S\subseteq A$ and $S\in\BI$.
\end{remark}

The next definition is fundamental for our considerations.

\begin{definition}
We say that $\I$ is boring if $\BI\sqsubseteq \I$. Otherwise we say that $\I$ is unboring.
\end{definition}

Throughout the entire paper we may use the next result without any reference.

\begin{proposition}
\label{q}
The following are equivalent for any ideal $\I$:
\begin{itemize}
\item[(i)] $\BI\not\leq_K \I$;
\item[(ii)] $\BI\not\leq_{KB} \I$;
\item[(iii)] $\BI\not\sqsubseteq \I$;
\item[(iv)] for any partition $(X_{(i,j)})$ of $\bigcup\I$ into sets belonging to $\I$ there is an $A\notin\I$ such that $A\cap X_{(i,j)}$ is finite for all $(i,j)\in\omega^2$ and $A\cap \bigcup_{j}X_{(i,j)}$ is finite for almost all $i\in\omega$;
\item[(v)] for each $f:\bigcup\I\to\omega^2$ there is an $A\notin\I$ such that $f[A]\in\fin^2$ and $f|A$ is either constant or finite-to-one.
\end{itemize}
\end{proposition}

\begin{proof}
Without loss of generality, we can assume that $\I$ is an ideal on $\omega$.

(i)$\implies$(ii): Obvious.

(ii)$\implies$(iii): Obvious.

(iii)$\implies$(iv): We will show that negation of (iv) implies $\BI\sqsubseteq\I$. Suppose that there is a partition $(X_{(i,j)})$ of $\omega$ into sets belonging to $\I$ such that if $A\cap X_{(i,j)}$ is finite for all $(i, j)\in\omega^2$ and $A\cap\bigcup_j X_{(i,j)}$ is finite for almost all $i\in\omega$ then $A\in\I$. Let $f : \omega\to\omega^3$ be any injective function satisfying $f[X_{(i,j)}]\subseteq\{(i, j)\}\times\omega$.

Observe that if $B\in\BI$ then $f^{-1}[B]\in\I$. Indeed, for each $B\in\BI$ we can find $k\in\omega$ such that $B_{(i)}\in\fin^2$ for all $i < k$ and $B_{(i)}$ is finite for all $i\geq k$. Thus, for every $i < k$ there is a $j_i\in\omega$ such that $B_{(i,j)}\in\fin$ for all $j>j_i$. Hence,
$$f^{-1}\left[B\cap\bigcup_{i<k}\bigcup_{j<j_i}(\{(i,j)\}\times\omega)\right]\subseteq\bigcup_{i<k}\bigcup_{j<j_i} X_{(i,j)}\in\I.$$
Moreover, $B'=B\setminus\bigcup_{i<k}\bigcup_{j<j_i}(\{(i,j)\}\times\omega)$ intersects each $\{(i,j)\}\times\omega$ in a finite set (so also $f^{-1}[B']\cap X_{(i,j)}\in\fin$ for all $(i, j)\in\omega^2$) and $B'\cap\{i\}\times\omega^2=B'\cap\bigcup_j(\{(i,j)\}\times\omega)$ is finite for all $i\geq k$ (so $f^{-1}[B']\cap\bigcup_j X_{(i,j)}$ is finite for almost all $i\in\omega$). Thus, $f^{-1}[B']\in\I$, by our assumption and we can conclude that
$$f^{-1}[B] = f^{-1}[B']\cup f^{-1}\left[B\cap\bigcup_{i<k}\bigcup_{j<j_i}(\{(i,j)\}\times\omega)\right]\in\I.$$

As $\BI$ is tall and $f$ is injective such that $f^{-1}[B]\in\I$ for all $B\in\BI$, using
\cite[Lemma 3.3]{Kat} we get that $\BI\sqsubseteq\I$. 

(iv)$\implies$(v): Let $f:\omega\to\omega^2$ be such that $f^{-1}[\{(i,j)\}]\in\I$ for all $(i,j)\in\omega^2$, i.e., $f|A$ being constant implies $A\in\I$ (if there is an $A\notin\I$ with $f|A$ constant, then we are done as $f[A]\in\fin(\omega^2)\subseteq\fin^2$). Consider the partition of $\omega$ given by $f^{-1}[\{(i,j)\}]$ for all $(i,j)\in\omega^2$. By (iv), there is an $A\notin\I$ such that there is a $k\in\omega$ with $A\cap f^{-1}[\{i\}\times\omega]\in\fin$ for all $i>k$ and $A\cap f^{-1}[\{(i,j)\}]\in\fin$ for all $i\leq k$ and $j\in\omega$. Thus, $f|A$ is finite-to-one (as $A\cap f^{-1}[\{(i,j)\}]\in\fin$ for all $i,j$) and $f[A]\in\fin^2$ (as for each $i>k$ there are only finitely many $j$ such that $A\cap f^{-1}[\{(i,j)\}]\neq\emptyset$).

(v)$\implies$(i): Assume to the contrary that $\BI\leq_K \I$ and let $g:\omega\to\omega^3$ be the witnessing function. Define $f:\omega\to\omega^2$ by $g(n)\in\{f(n)\}\times\omega$. Then there is an $A\notin\I$ such that $f[A]\in\fin^2$ and $f|A$ is either constant or finite-to-one. However, $f|A$ being constant implies that $A\subseteq g^{-1}[\{(i,j)\}\times\omega]\in\I$ for some $(i,j)\in\omega$. Thus, $f|A$ is finite-to-one. But then $A\cap g^{-1}[\{(i,j)\}\times\omega]$ is finite for all $(i,j)\in\omega$ and, as $f[A]\in\fin^2$, there is a $k\in\omega$ such that $A\cap g^{-1}[\{i\}\times\omega^2]=A\cap f^{-1}[\{i\}\times\omega]\in\fin$ for all $i>k$. As $g$ witnesses $\BI\leq_K \I$, we conclude that $A\in\I$, which contradicts the choice of $A$ and finishes the proof.
\end{proof}

\begin{remark}
In the terminology of \cite{Kat}, equivalence of items (i) and (iii) from Proposition \ref{q} means that $\BI$ has the property Kat.
\end{remark}

\begin{proposition}
\label{conv}
We have $\conv\sqsubseteq\BI\sqsubseteq\fin^2$ and none of those inclusions can be reversed.
\end{proposition}

\begin{proof}
It is easy to check that $\BI\sqsubseteq\fin^2$ is witnessed by any bijection $f:\omega^2\to\omega^3$ satisfying $f[\{n\}\times\omega]=\{n\}\times\omega^2$. 

By Proposition \ref{propertyKat}, in order to show $\conv\sqsubseteq\BI$, it suffices to prove that $\conv\leq_K\BI$, which in turn is equivalent to existence of a countable family $\{X_n:\ n\in\omega\}$ such that for each $A\notin\BI$ there is an $n\in\omega$ with $|A\cap X_n|=|A\setminus X_n|=\omega$ (by \cite[Theorem 2.4.3]{Meza}). We claim that $\{\{(i,j)\}\times\omega:\ i,j\in\omega\}\cup\{\{i\}\times\omega^2:\ i\in\omega\}$ is the required family. Indeed, fix $A\notin\BI$ and assume first that $A\cap \{(i,j)\}\times\omega$ is infinite for some $i,j\in\omega$. Then $A\setminus \{(i,j)\}\times\omega$ has to be infinite as well since $\{(i,j)\}\times\omega\in\BI$. Thus, $\{(i,j)\}\times\omega$ is the required set. Assume now that $A\cap \{(i,j)\}\times\omega$ is finite for all $i,j\in\omega$. Then $A\cap \{i\}\times\omega^2$ is infinite for infinitely many $i\in\omega$ (as $A\notin\BI$). Thus, we can find $i\in\omega$ such that $|A\cap (\{i\}\times\omega^2)|=|A\setminus (\{i\}\times\omega^2)|=\omega$.

Now we show that $\BI\not\sqsubseteq\conv$, i.e., that $\conv$ is unboring. We will use item (v) of Proposition \ref{q}. Fix $f:\mathbb{Q}\cap[0,1]\to\omega^2$. If $f^{-1}[\{(i,j)\}]\notin\I$ for some $(i,j)\in\omega^2$ then put $A=f^{-1}[\{(i,j)\}]$ and observe that $f[A]\in\fin^2$, $A\notin\I$ and $f|A$ is constant. On the other hand, if $f^{-1}[\{(i,j)\}]\in\I$ for each $(i,j)\in\omega^2$, then fix any bijection $h:\omega\to\omega^2$, enumerate all basic open sets in $[0,1]$ as $\{U_n:\ n\in\omega\}$ and pick $q_n\in U_n\setminus\bigcup_{i<n}f^{-1}[\{h(i)\}]$ for all $n\in\omega$. This can be done as $U_n\cap\mathbb{Q}\notin\conv$, for each $n\in\omega$. Then $B=\{q_n:\ n\in\omega\}\notin\conv$ and $f|B$ is finite-to-one. 

Observe that $\fin^2\not\sqsubseteq\conv|B$. Indeed, given any bijection $g:B\to\omega^2$ such that $g^{-1}[\{n\}\times\omega]\in\conv|B$, for all $n\in\omega$, one can inductively pick $r_n\in (U_n\cap B)\setminus(\bigcup_{i<n}g^{-1}[\{i\}\times\omega])$ (as $U_n\cap B\notin\conv|B$). Then $D=\{r_n:\ n\in\omega\}\subseteq B$, $D\notin\conv|B$ and $g[D]\in\emptyset\otimes\fin\subseteq\fin^2$.

By Proposition \ref{propertyKat}, $\fin^2\not\leq_{KB}\conv|B$, so there is a $C\in\fin^2$ such that $E=B\cap f^{-1}[C]\notin\conv|B$. Then $f[E]\in\fin^2$ and $f|E$ is finite-to-one (as $E\subseteq B$). Hence, $\conv$ is unboring.

Finally, we show that $\fin^2\not\sqsubseteq\BI$. Let $f:\omega^3\to\omega^2$ be any bijection. Without loss of generality we can assume that $f^{-1}[\{i\}\times\omega]\in\BI$ for all $i\in\omega$ (otherwise we are done). Fix any $g:\omega\to\omega$ such that $g^{-1}[\{n\}]$ is infinite for all $n\in\omega$. Pick inductively points $x_n\in\omega^3$ in such a way that
$$x_n\in(\{g(n)\}\times\omega^2)\setminus\bigcup\{f^{-1}[\{i\}\times\omega]:\ (\exists j<n)\ f(x_j)\in\{i\}\times\omega\}$$
(this is possible as each set of the form $\{g(n)\}\times\omega^2$ does not belong to $\BI$). Then $X=\{x_n:\ n\in\omega\}\notin\BI$ (as $X_{(n)}$ is infinite for all $n\in\omega$), but $f[X]\in\emptyset\otimes\fin\subseteq\fin^2$. 
\end{proof}

\begin{corollary}
Let $\I$ be a maximal ideal. Then $\I^\star$ is a P-point if and only if $\I$ is unboring.
\end{corollary}

\begin{proof}
Follows from Propositions \ref{P-point} and \ref{conv}.
\end{proof}

\subsection{Hereditary weak P-ideals}

The following definition will play a very important role in our considerations.

\begin{definition}
An ideal $\I$ is a hereditary weak P-ideal if $\fin^2\not\sqsubseteq\I|A$ for each $A\notin\I$.
\end{definition}

An ideal $\I$ is a P$^+$-ideal if for every decreasing sequence $(A_n)$ with $A_n\notin\I$, for all $n\in\omega$, there is an $A\notin \I$ such that $A\setminus A_n$ is finite for each $n\in \omega$. It is known that each $\bf{\Sigma^0_2}$ ideal is a P$^+$-ideal (cf. \cite[Proposition 5.1]{I-selection}).

We say that an ideal $\I$ on $X$ is weakly selective if for each partition $(A_n)$ of $X$ such that $A_0\notin\I^\star$ and $A_{n+1}\in\I$, for all $n\in\omega$, there is an $S\notin\I$ with $|S\cap A_n|\leq 1$, for all $n\in\omega$. It is easy to see that $\I$ is weakly selective if and only if $\mathcal{ED}\not\leq_K\I|A$ for all $A\notin\I$ \cite[Section 2.7]{Meza}, where 
$$\ED=\left\{A\subseteq\omega^2:\ \left(\exists k,m\in\omega\right)\ \left(\forall i>k\right)\ |A_{(i)}|<m\right\}.$$

\begin{proposition}
\label{all-cor}
Suppose that $\I$ satisfies at least one of the following conditions: 
\begin{itemize}
\item[(a)] $\I$ is $\bf{\Pi^0_4}$;
\item[(b)] $\I$ is a P$^+$-ideal.
\item[(c)] $\I$ is weakly selective.
\end{itemize}
Then $\I$ is a hereditary weak P-ideal.
\end{proposition}

\begin{proof}
(a): Observe that $\I|B$ is $\bf{\Pi^0_4}$, for all $B\notin\I$ (as the identity function from $\mathcal{P}(B)$ to $\mathcal{P}(\bigcup\I)$ is continuous). By \cite[Theorems 7.5 and 9.1]{Debs}, a $\bf{\Pi^0_4}$ ideal cannot contain an isomorphic copy of $\fin^2$.

(b): Follows from \cite[Theorem 2.4.2]{Meza}.

(c): Since $\I$ is weakly selective, $\ED\not\leq_{K}\I|A$ for each $A\notin\I$. It follows that $\fin^2\not\sqsubseteq\I|A$ for each $A\notin\I$ (otherwise we would get $\ED\sqsubseteq\fin^2\sqsubseteq\I|A$).
\end{proof}

\begin{proposition}
\label{boringprop}
The following hold for any ideal $\I$:
\begin{enumerate}
\item[(a)] $\I$ is a hereditary weak P-ideal $\implies$ $\I$ is unboring $\implies$ $\I$ is a weak P-ideal (equivalently, $\fin^2\sqsubseteq\I\implies\BI\sqsubseteq\I\implies(\exists A\notin\I)\ \fin^2\sqsubseteq\I|A$).
\item[(b)] Both implications from item (a) cannot be reversed.
\item[(c)] If $\I$ is extendable to an unboring ideal then $\I$ is unboring. In particular, if $\I$ is extendable to hereditary weak P-ideal then $\I$ is unboring.
\end{enumerate}
\end{proposition}

\begin{proof}
(a): If $\fin^2\sqsubseteq\I$ then $\mathcal{BI}\sqsubseteq\I$ by Proposition \ref{conv}. This shows that $\I$ being unboring implies that $\I$ is a weak P-ideal.

To show the first implication, let $\I$ be a hereditary weak P-ideal. We need to show that $\I$ is unboring.

We will use item (v) of Proposition \ref{q}. Fix $f:\omega\to\omega^2$. If $f^{-1}[\{(i,j)\}]\notin\I$ for some $(i,j)\in\omega^2$ then put $A=f^{-1}[\{(i,j)\}]$ and observe that $f[A]\in\fin^2$, $A\notin\I$ and $f|A$ is constant. On the other hand, if $f^{-1}[\{(i,j)\}]\in\I$ for each $(i,j)\in\omega^2$, then $\{f^{-1}[\{(i,j)\}]:\ (i,j)\in\omega^2\}$ defines a partition of $\omega$ into sets from $\I$. As $\fin^2\not\sqsubseteq\I$, there is a $B\notin\I$ such that $f|B$ is finite-to-one. Moreover, since $\fin^2\not\leq_{KB}\I|B$ (by Proposition \ref{propertyKat}), there is a $C\in\fin^2$ such that $A=B\cap f^{-1}[C]\notin\I|B$. Then $A\notin\I$, $f[A]\in\fin^2$ and $f|A$ is finite-to-one.

(b): To show that $\fin^2\sqsubseteq\I\implies\BI\sqsubseteq\I$ cannot be reversed, it suffices to consider the ideal $\BI$, which is boring, but $\fin^2\not\sqsubseteq\BI$ (by Proposition \ref{conv}).

Now we show that the other implication cannot be reversed. The ideal $\conv$ is a good example. By Proposition \ref{conv}, $\conv$ is unboring. On the other hand, for each $n\in\omega$ pick a sequence $(x_{n,k})_k\subseteq\mathbb{Q}\cap(\frac{1}{n+2},\frac{1}{n+1})$ converging to $\frac{1}{n+2}$. Then $\fin^2\sqsubseteq\conv|A$, where $A=\{x_{n,k}:\ n,k\in\omega\}$ (as witnessed by the function $f:A\to\omega^2$ given by $f(x_{n,k})=(n,k)$).

(c): Let $\J$ be an unboring ideal containing $\I$. Then $\I$ cannot be boring as otherwise we would get $\BI\sqsubseteq\I\subseteq\J$, which contradicts the fact that $\J$ is unboring. The second part follows from item (a).
\end{proof}

By the above result, if $\I$ is extendable to a hereditary weak P-ideal, then $\I$ is unboring. In the next subsection we show that this implication cannot be reversed even for Borel ideals.

\begin{proposition}
\label{hh}
$\I$ is a hereditary weak P-ideal if and only if $\I$ is hereditary unboring (that is, $\BI\sqsubseteq\I|B$ for no $B\notin\I$).
\end{proposition}

\begin{proof}
If $\I$ is a hereditary unboring ideal then the existence of $B\notin\I$ with $\fin^2\sqsubseteq\I|B$ would imply $\BI\sqsubseteq\I|B$ (by Proposition \ref{conv}). Thus, $\I$ is a hereditary weak P-ideal.

Conversely, assume that $\I$ is a hereditary weak P-ideal and fix $B\notin\I$. If $\BI\sqsubseteq\I|B$ and $f:B\to\omega^3$ is the witnessing bijection, then either $f^{-1}[\{i\}\times\omega^2]\in\I$ for all $i\in\omega$ or $f^{-1}[\{n\}\times\omega^2]\notin\I$ for some $n\in\omega$. In the former case, $(f^{-1}[\{i\}\times\omega^2])_{i\in\omega}$ is a partition of $B$ into sets belonging to $\I$ such that each $X\subseteq B$ with $X\cap f^{-1}[\{i\}\times\omega^2]\in\fin(B)$ for all $i\in\omega$ is in $\I$, i.e., $\fin^2\sqsubseteq\I|B$. In the latter case, $(f^{-1}[\{(n,i)\}\times\omega^2])_{i\in\omega}$ is a partition of $C=f^{-1}[\{n\}\times\omega^2]\notin\I$ into sets belonging to $\I$ such that each $X\subseteq C$ with $X\cap f^{-1}[\{(n,i)\}\times\omega^2]\in\fin(C)$ for all $i\in\omega$ is in $\I$, i.e., $\fin^2\sqsubseteq\I|C$. Thus, both cases contradict the assumption that $\I$ is a hereditary weak P-ideal.
\end{proof}

\begin{corollary}
\label{K-uniform}
The following are equivalent for any $\leq_{K}$-uniform ideal $\I$:
\begin{itemize}
\item[(a)] $\I$ is a weak P-ideal;
\item[(b)] $\I$ is unboring;
\item[(c)] $\I$ is a hereditary weak P-ideal;
\item[(d)] $\I$ is hereditary unboring (that is, $\BI\sqsubseteq\I|B$ for no $B\notin\I$).
\end{itemize}
\end{corollary}

\begin{proof}
(c)$\Longleftrightarrow$(d): This is Proposition \ref{hh}.

(c)$\implies$(b) and (b)$\implies$(a): This is item (a) of Proposition \ref{boringprop}.

(a)$\implies$(c): Suppose to the contrary that $\fin^2\sqsubseteq\I|A$ for some $A\notin\I$. Since $\I$ is $\leq_{K}$-uniform, we have $\fin^2\sqsubseteq\I|A\leq_K \I$ which is equivalent to $\I$ not being a weak P-ideal (by Propositions \ref{o} and \ref{propertyKat}). Thus, $\fin^2\not\sqsubseteq\I|A$ for all $A\notin\I$.
\end{proof}

We end this subsection with an observation about extendability to hereditary weak P-ideals.

\begin{proposition}
\label{smallest}
If $\I$ is extendable to a hereditary weak P-ideal then there is a smallest (in the sense of inclusion) hereditary weak P-ideal $\widehat{\I}$ containing $\I$.
\end{proposition}

\begin{proof}
Define $\widehat{\I}=\bigcap\{\J:\ \J\text{ is a hereditary weak P-ideal containing }\I\}$. Then $\I\subseteq\widehat{\I}$ and $\widehat{\I}$ is an ideal as an intersection of ideals. We need to show that $\widehat{\I}$ is a hereditary weak P-ideal. Let $B\notin\widehat{\I}$. Then there is a hereditary weak P-ideal $\J$ containing $\I$ with $B\notin\J$. Thus, $\fin^2\sqsubseteq\widehat{\I}|B$ would imply $\fin^2\sqsubseteq\widehat{\I}|B\subseteq\J|B$ which contradicts the fact that $\J$ is a hereditary weak P-ideal. 
\end{proof}

\subsection{A technical notion}

Following \cite{Laflamme} we say that an ideal $\I$ on $X$ is $\omega$-diagonalizable by $\mathcal{I}^\star$-universal sets, if one can find a family $\{\cZ_k:\ k\in\omega\}$ such that:
\begin{itemize} 
\item for each $k\in\omega$ the set $\cZ_k\subseteq [X]^{<\omega}\setminus \{\emptyset\}$ is $\mathcal{I}^\star$-universal, which means that for each $A\in \I$ there is a $Z\in \cZ_k$ with $Z\cap A=\emptyset$;
\item for each $A\in \I$ there is a $k\in\omega$ such that $Z\not\subseteq A$ for every $Z\in \cZ_k$. 
\end{itemize}

The following notion will be useful in considerations about replacing CH with MA($\sigma$-centered) in some of our results.

\begin{definition}
We say that an ideal $\I$ is strongly unboring if for each $f\in\cD_\I$ there is a $C\notin\I$ such that $f|C$ is finite-to-one and $\I|C$ is $\omega$-diagonalizable by $(\mathcal{I}|C)^\star$-universal sets.
\end{definition}

\begin{lemma}[{Essentially \cite{Reclaw} and \cite{Laflamme}}]
\label{Laflamme}
If $\I$ is a coanalytic ideal and $\fin^2\not\sqsubseteq\I$ then $\I$ is $\omega$-diagonalizable by $\mathcal{I}^\star$-universal sets.
\end{lemma}

\begin{proof}
Consider the game $G\left(\I\right)$, defined by Laflamme (see \cite{Laflamme}) as follows: Player I in his $n$'th move plays an element $C_{n}\in\I$, and then Player II responds with any $F_{n}\in \left[\bigcup\I\right]^{<\omega}$ such that $F_{n}\cap C_{n}=\emptyset$. Player I wins if $\bigcup_{n\in \omega} F_{n}\in \I$. Otherwise, Player II wins. By \cite[Theorem 5.1]{equal}, $G\left(\I\right)$ is determined as $\I$ is a coanalytic ideal. By \cite[Theorem 2.16]{Laflamme}, Player I has a winning strategy in $G\left(\I\right)$ if and only if $\fin^2\sqsubseteq\I$. Thus, Player II has to have a winning strategy. Again by \cite[Theorem 2.16]{Laflamme}, this is in turn equivalent to $\I$ being $\omega$-diagonalizable by $\mathcal{I}^\star$-universal sets.
\end{proof}

The above was first observed by M. Laczkovich and I. Rec\l{}aw in \cite{Reclaw}. However, their observation concerned Borel ideals. In \cite{equal} it was shown, using an idea from \cite{Sabok}, that this is also true for coanalytic ideals. See also \cite{3kinds} for some applications of this result.

\begin{proposition}
\label{r}
The following hold for any ideal $\I$:
\begin{enumerate}
\item[(a)] If $\I$ is extendable to a strongly unboring ideal then $\I$ is strongly unboring.
\item[(b)] $\I$ is extendable to a coanalytic hereditary weak P-ideal $\implies$ $\I$ is strongly unboring $\implies$ $\I$ is unboring.
\item[(c)] There is a Borel strongly unboring (so also unboring, by the previous item) ideal not extendable to a hereditary weak P-ideal. In particular, the first implication from item (b) cannot be reversed even for Borel ideals.
\item[(d)] If $\cU$ is a P-point then $\cU^\star$ is an unboring ideal which is not strongly unboring. In particular, under MA($\sigma$-centered) there are unboring ideals which are not strongly unboring.
\end{enumerate}
\end{proposition}

\begin{proof}
(a): This is easy to verify using the definition of strongly unboring ideals.

(b): At first we will prove the first implication. Let $\J$ be a coanalytic hereditary weak P-ideal containing $\I$. We will show that $\J$ is strongly unboring. By item (a), this will finish the proof.

Fix $f\in\cD_\J$. Since $(f^{-1}[\{n\}])$ defines a partition of $\bigcup\J$ into sets belonging to $\J$ and $\fin^2\not\sqsubseteq\J$, there is a $C\notin\J$ such that $f|C$ is finite-to-one. Note that $\J|C$ is a coanalytic ideal (as $\J$ is coanalytic and the identity function from $\mathcal{P}(C)$ to $\cP(\bigcup\J)$ is continuous). Since $\J$ is a hereditary weak P-ideal, $\fin^2\not\sqsubseteq\J|C$. Thus, by Lemma \ref{Laflamme}, $\J|C$ is $\omega$-diagonalizable by $(\mathcal{J}|C)^\star$-universal sets. Hence, $C$ is the required set.

Now we show the second implication. Suppose that $\I$ is boring, i.e., $\BI\sqsubseteq\I$. Without loss of generality we may assume that $\bigcup\I=\omega^3$ and $\BI\subseteq\I$ (by considering an appropriate isomorphic copy of $\I$). Fix any bijection $h:\omega\to\omega^2$ and define $f:\bigcup\I\to\omega$ by $f[\{h(n)\}\times\omega]=\{n\}$. Then $f\in\cD_{\BI}\subseteq\cD_\I$. Let $B\notin\I$ be such that $f|B$ is finite-to-one. Denote by $C=\{(i,j)\in\omega^2:\ B\cap(\{(i,j)\}\times\omega)\neq\emptyset\}$. Then $C\notin\fin^2$ (as otherwise we would have $B\in\BI\subseteq\I$). As $\fin^2$ is homogeneous (see Example \ref{uniform-ex}), $\fin^2|C\cong\fin^2$. Let $g:C\to\omega^2$ be the witnessing isomorphism and denote by $\pi_{1,2}:\omega^3\to\omega^2$ the projection onto the first two coordinates (i.e., $\pi_{1,2}(x,y,z)=(x,y)$ for all $(x,y,z)\in\omega^3$). Then $g\circ \pi_{1,2}|B$ witnesses that $\fin^2\leq_{KB}\BI|B\subseteq\I|B$. By Proposition \ref{propertyKat}, this means that $\fin^2\sqsubseteq\I|B$. Using the witnessing bijection it is easy to check that $\I|B$ being $\omega$-diagonalizable by $(\mathcal{I}|B)^\star$-universal sets would imply that $\fin^2$ is $\omega$-diagonalizable by $(\fin^2)^\star$-universal sets. However, this is impossible by \cite[Theorem 6.2]{3kinds} (see also \cite[Theorem 7.5]{Debs} or \cite[Theorem 5]{Reclaw}). Hence, $\I|B$ cannot be $\omega$-diagonalizable by $(\mathcal{I}|B)^\star$-universal sets. Therefore, each boring ideal is not strongly unboring.

(c): Consider the ideal $\I$ on $\omega^4$ generated by:
\begin{itemize}
\item sets $\{(i,j,k)\}\times\omega$, for all $(i,j,k)\in\omega^3$;
\item sets $B\subseteq\omega^4$ such that $B\cap(\{(i,j,k)\}\times\omega)$ is finite for all $(i,j,k)\in\omega^3$ and
$$\left\{(i,j,k)\in\omega^3:\ B\cap(\{(i,j,k)\}\times\omega)\neq\emptyset\right\}\in\BI.$$
\end{itemize}
Equivalently, $\I=(\BI\otimes\emptyset)\cap\fin(\omega^3)\otimes\fin$. It is easy to see that $\I$ is Borel.

First we show that $\I$ is not extendable to a hereditary weak P-ideal. Suppose otherwise, i.e., that there is a hereditary weak P-ideal $\I'\supseteq\I$. If $\{(i,j)\}\times \omega^2\notin\I'$ for some $(i,j)\in\omega^2$, then $\I'_{i,j}=\{A\subseteq\omega^4:\ A\cap(\{(i,j)\}\times\omega^2)\in\I'\}$ is an ideal on $\omega^4$. It is easy to see that $\I'$ being a hereditary weak P-ideal implies that $\I'_{i,j}$ is a hereditary weak P-ideal as well. However, $\I'_{i,j}\supseteq\{A\subseteq\omega^4:\ A\cap(\{(i,j)\}\times\omega^2)\in\I\}\cong\fin^2$ and $\fin^2$ is not extendable to a hereditary weak P-ideal (by Propositions \ref{conv} and \ref{boringprop}(c)). Therefore $\{(i,j)\}\times \omega^2\in\I'$ for all $(i,j)\in\omega^2$. However, this implies that $\I'$ is boring (as witnessed by any bijection $f:\omega^4\to\omega^3$ with $f[\{(i,j)\}\times\omega^2]=\{(i,j)\}\times\omega$). Recall that a hereditary weak P-ideal cannot be boring (by Proposition \ref{boringprop}(a)). Thus, we get a contradiction which proves that $\I$ is not extendable to a hereditary weak P-ideal.

Now we show that $\I$ is strongly unboring. Fix $f\in\cD_\I$ and a bijection $h:\omega\to\omega^3$. Inductively pick points $x_n\in\omega^4$ such that:
\begin{itemize}
\item $x_n\in\{h(n)\}\times\omega$;
\item if $f[\{h(n)\}\times\omega]\setminus f[\{x_i:\ i<n\}]$ is nonempty then $f(x_n)$ belongs to that set;
\item if $f[\{h(n)\}\times\omega]\subseteq f[\{x_i:\ i<n\}]$ then $f(x_n)=\max f[\{x_i:\ i<n\}]$.
\end{itemize}
Then $X=\{x_n:\ n\in\omega\}\notin\I$ since $|X\cap(\{(i,j,k)\}\times\omega)|=1$ for all $(i,j,k)\in\omega^3$. Moreover, we have $\fin^2\not\sqsubseteq\I|X$ (by Proposition \ref{conv}, as $\BI$ and $\I|X$ are isomorphic). Since $\I$ is Borel, so is $\I|X$ (as the identity function from $X$ to $\omega^3$ is continuous). Thus, by Lemma \ref{Laflamme}, $\I|X$ is $\omega$-diagonalizable by $(\mathcal{I}|X)^\star$-universal sets. To finish the proof we only need to show that $f|X$ is finite-to-one. 

Fix $n\in\omega$. By the construction of $X$, $X\cap f^{-1}[\{n\}]$ can be infinite only if there were infinitely many $m\in\omega$ with $f[\{h(m)\}\times\omega]\subseteq \{0,1,\ldots,n\}$, i.e., $(\{h(m)\}\times\omega)\subseteq f^{-1}[\{0,1,\ldots,n\}]$. However, $f^{-1}[\{0,1,\ldots,n\}]\in\I\subseteq\fin(\omega^3)\otimes\fin$. Consequently, there are only finitely many $m\in\omega$ covered by $f^{-1}[\{0,1,\ldots,n\}]$. Thus, $f|X$ is finite-to-one. 

(d): Let $\mathcal{U}$ be a P-point. Denote $\I=\mathcal{U}^\star$. By Proposition \ref{P-point}, $\fin^2\not\sqsubseteq\I$. Since $\I$ is a maximal ideal, it is homogeneous (see Example \ref{uniform-ex}). Thus, $\I$ is a hereditary weak P-ideal. This means that it is unboring (by Proposition \ref{boringprop}(a)). On the other hand, we will show that for each $C\notin\I$ the ideal $\I|C$ is not $\omega$-diagonalizable by $(\mathcal{I}|C)^\star$-universal sets (hence $\I$ cannot be strongly unboring). Fix any $C\notin\I$. By the proof of \cite[Theorem 4]{Reclaw}, $\I|C$ being $\omega$-diagonalizable by $(\mathcal{I}|C)^\star$-universal sets means that there is a $\bf{\Sigma^0_2}$ set $S$ with $\I|C\subseteq S$ and $S\cap(\I|C)^\star=\emptyset$. Since $\I$ is homogeneous, $\I|C\cong\I$. In particular, $\I|C$ is a maximal ideal on $C$, so $\I|C\cup(\I|C)^\star=\cP(C)$and $S=\I|C$. However, $\I|C$ is not Borel (as a maximal ideal). This contradiction finishes the proof.
\end{proof}

\section{Mr\'{o}wka spaces}

\subsection{Results in ZFC}

If $\mathcal{A}\subseteq[\omega]^\omega$ is an AD family then by $\fin^2(\mathcal{A})$ we denote the ideal on $\omega$ generated by sets belonging to $\mathcal{A}$ and by sets having finite intersection with each member of $\mathcal{A}$. 

\begin{proposition}
\label{MrowkaChar}
The following are equivalent for any ideal $\I$ and any AD family $\mathcal{A}$:
\begin{itemize}
\item[(a)] $\Phi(\mathcal{A})$ is in FinBW($\I$);
\item[(b)] for every function $f\in\mathcal{D}_\I$ there is a $B\notin\I$ such that $f|B$ is finite-to-one and $f[B]\in\fin^2(\mathcal{A})$.
\end{itemize}
\end{proposition}

\begin{proof}
Without loss of generality we may assume that $\I$ is an ideal on $\omega$.

(a)$\implies$(b): Let $X=\Phi(\mathcal{A})$ be in FinBW($\I$). Assume to the contrary that (b) does not hold. Then there is an $f:\omega\to\omega$ satisfying $f^{-1}[\{n\}]\in\I$ for all $n\in\omega$ such that for each $B\notin\I$ with $f|B$ finite-to-one we have $f[B]\cap A\notin\fin$ for infinitely many $A\in\mathcal{A}$. We will show that the sequence $(f(n))\subseteq X$ does not possess a convergent subsequence indexed by a set not belonging to $\I$.

Suppose that $B\notin\I$ and consider $(f(n))_{n\in B}$. Since $f^{-1}[\{n\}]\in\I$ for all $n\in\omega$, $(f(n))_{n\in B}$ cannot converge to any $n\in\omega$. Moreover, $(f(n))_{n\in B}$ cannot converge to any $A\in\mathcal{A}$, as there is an $A'\in\mathcal{A}$, $A\neq A'$ with $|f[B]\cap (A'\setminus A)|=|f[B]\cap A'|=\omega$. Finally, since $f[B]\cap A\notin\fin$ for some $A\in\mathcal{A}$, $(f(n))_{n\in B}$ cannot converge to $\infty$ (as $U=X\setminus (\{A\}\cup A)$ is an open neighborhood of $\infty$ such that infinitely many elements of the sequence $(f(n))_{n\in B}$ are outside $U$).

(b)$\implies$(a): We need to show that $X=\Phi(\mathcal{A})$ is in FinBW($\I$). Fix any $f:\omega\to X$. Without loss of generality we can assume that $f^{-1}[\{x\}]\in\I$ for all $x\in X$ (otherwise we are done as we would get a constant, so converging, sequence indexed by an $\I$-positive set). We have two possibilities: either $\fin^2\not\sqsubseteq \I|f^{-1}[X\setminus\omega]$ or $\fin^2\sqsubseteq \I|f^{-1}[X\setminus\omega]$. 

In the former case, we can find $B\notin\I$ such that $f[B]\subseteq X\setminus\omega$ and $f|B$ is finite-to-one. We claim that $(f(n))_{n\in B}$ converges to $\infty$. Indeed, let $U$ be a basic neighborhood of $\infty$. Then $(X\setminus\omega)\setminus U$ is finite. Since $f|B$ is finite-to-one and $f[B]\subseteq X\setminus\omega$, almost all elements of $(f(n))_{n\in B}$ have to belong to $U$.

In the latter case, there is a bijection $h:f^{-1}[X\setminus\omega]\to\omega^2$ witnessing $\fin^2\sqsubseteq \I|f^{-1}[X\setminus\omega]$. Define $\widetilde{h}:f^{-1}[X\setminus\omega]\to\omega$ by $\widetilde{h}(n)=\pi_1(h(n))$, for all $n\in f^{-1}[X\setminus\omega]$, where $\pi_1:\omega^2\to\omega$ is the projection onto the first coordinate. Let $g:\omega\to\omega$ be given by $g(n)=f(n)$, for all $n\in f^{-1}[\omega]$, and $g(n)=\widetilde{h}(n)$ for all $n\in f^{-1}[X\setminus\omega]$. Then $g^{-1}[\{n\}]\subseteq h^{-1}[\{n\}\times\omega]\cup f^{-1}[\{n\}]\in\I$, so there is a $B\notin\I$ such that $g|B\in\mathcal{D}_{\fin(B)}$ and $g[B]\in\fin^2(\mathcal{A})$. Define $C=B\cap f^{-1}[\omega]$. Then $f|C=g|C$ is finite-to-one and $f[C]=g[C]\in\fin^2(\mathcal{A})$. 

We claim that $C\notin\I$. Indeed, denote $D=B\cap f^{-1}[X\setminus\omega]$. Since $g|D\in\mathcal{D}_{\fin(D)}$ and $D\subseteq f^{-1}[X\setminus\omega]$, the function $\widetilde{h}|D$ is finite-to-one. This implies $h[D]\in\emptyset\otimes\fin\subseteq\fin^2$ and hence $D\in\I|f^{-1}[X\setminus\omega]$ (by the choice of $h$). As $C\cup D=B\notin\I$, we get that $C\notin\I$.

Since $C\notin\I$ and $f[C]\in\fin^2(\mathcal{A})$, there are two possibilities:
\begin{itemize}
\item either there are $C'\subseteq C$, $C'\notin\I$ and $A\in\mathcal{A}$ with $f[C']\subseteq A$ (in this case $(f(n))_{n\in C'}$ converges to $A$);
\item or there is a $C'\subseteq C$, $C'\notin\I$ with $f[C']\cap A\in\fin$ for all $A\in\mathcal{A}$ (in this case $(f(n))_{n\in C'}$ converges to $\infty$).
\end{itemize}
This finishes the proof.
\end{proof}

\begin{lemma}
\label{MAD}
Let $\I$ be an ideal. If $\mathcal{A}$ is an AD family such that for every $f\in\mathcal{D}_\I$ there is a $B\notin\I$ such that $f|B$ is finite-to-one and $f[B]\subseteq A$ for some $A\in\mathcal{A}$, then $\cA$ is a MAD family. 
\end{lemma}

\begin{proof}
Suppose otherwise and let $C\subseteq\omega$ be infinite with $C\cap A\in\fin$ for all $A\in\mathcal{A}$. Fix any bijection $f:\bigcup\I\to C$. Then $f\in\mathcal{D}_{\fin(\bigcup\I)}\subseteq\mathcal{D}_\I$, but for any $B\subseteq\bigcup\I$ the condition $f[B]\subseteq A$, for some $A\in\mathcal{A}$, implies that $B\in\fin(\bigcup\I)\subseteq\I$. Thus, the function $f$ would contradict the assumptions of this lemma.
\end{proof}

\subsection{Consequences of CH and MA}

\begin{theorem}
\label{Mrowka}
(CH) If $\I$ is unboring, then there is an uncountable Mr\'{o}wka space in FinBW($\I$).
\end{theorem}

\begin{proof}
By Proposition \ref{MrowkaChar} and Lemma \ref{MAD}, it suffices to show that there is an AD family $\mathcal{A}$ such that for every $f\in\mathcal{D}_\I$ there is a $B\notin\I$ such that $f|B$ is finite-to-one and $f[B]\subseteq A$ for some $A\in\mathcal{A}$. 

Fix a list $\{f_\alpha:\ \omega\leq\alpha<\omega_1\}=\mathcal{D}_\I$. We will construct a family $\mathcal{A}=\{A_\alpha:\ \alpha<\omega_1\}$ by induction on $\alpha$. The enumeration may contain repetitions. At first let $\{A_n:\ n\in\omega\}$ be any partition of $\omega$ into infinite sets.

Suppose that $\omega\leq\alpha<\omega_1$ and that $A_\beta$, for all $\beta<\alpha$, have been chosen. Using CH, enumerate $\alpha$ as $(\beta_i)_{i\in\omega}$ and consider the partition $(B_i)$ of $\omega$, where $B_0=A_{\beta_0}$ and $B_i=A_{\beta_i}\setminus\bigcup_{j<i}A_{\beta_j}$, for all $0<i<\omega$. Let $h:\omega\to\omega^2$ be any injection satisfying $h[B_i]\subseteq\{i\}\times\omega$. Define $f:\omega\to\omega^2$ by $f=h\circ f_\alpha$. Since $\I$ is unboring, there is a $C\notin\I$ such that $f[C]\in\fin^2$ and $f|C$ is either constant or finite-to-one (by Proposition \ref{q}(v)). 

Observe that $f|C$ cannot be constant, as $f_\alpha^{-1}[\{n\}]\in\I$ for all $n\in\omega$ and $h$ is an injection. Thus, $f|C$ is finite-to-one. Since $f[C]\in\fin^2$ and $C\notin\I$, either $C\cap f^{-1}[\{n\}\times\omega]\notin\I$ for some $n\in\omega$ or $C\cap f^{-1}[D]\notin\I$ for some $D\in\emptyset\otimes\fin$. In the former case put $B=C\cap f^{-1}[\{n\}\times\omega]$ and $A_\alpha=A_{\beta_n}$. Then $B\notin\I$, $f_\alpha|B$ is finite-to-one and $f_\alpha[B]\subseteq B_n\subseteq A_\alpha$. In the latter case put $B=C\cap f^{-1}[D]$ and $A_\alpha=f_\alpha[B]$. Then $B\notin\I$, $f_\alpha|B$ is finite-to-one and $f_\alpha[B]\subseteq A_\alpha$. Moreover, $A_\alpha\cap A_{\beta_i}$ is finite for each $i\in\omega$ as $D\in\emptyset\otimes\fin$.
\end{proof}

\begin{theorem}
\label{MAlem}
(MA($\sigma$-centered)) If $\I$ is strongly unboring, then there is an uncountable separable Mr\'{o}wka space in FinBW($\I$).
\end{theorem}

\begin{proof}
By Proposition \ref{MrowkaChar} and Lemma \ref{MAD}, it suffices to show that there is an AD family $\mathcal{A}$ such that for every $f\in\mathcal{D}_\I$ there is a $B\notin\I$ such that $f|B$ is finite-to-one and $f[B]\subseteq A$ for some $A\in\mathcal{A}$.

Let $\{f_\alpha:\ \omega\leq\alpha<\cc\}$ be an enumeration of $\mathcal{D}_\I$. We will construct a family $\mathcal{A}=\{A_\alpha:\ \alpha<\cc\}$ by induction on $\alpha$. The enumeration may contain repetitions. At first let $\{A_n:\ n\in\omega\}$ be any partition of $\omega$ into infinite sets.

Suppose that $\omega\leq\alpha<\cc$ and that $A_\beta$ for all $\beta<\alpha$ have been chosen. Consider the function $f_\alpha$. Since $\I$ is strongly unboring, we can find $C\notin\I$ such that $f_\alpha|C$ is finite-to-one and $\I|C$ is $\omega$-diagonalizable by $(\mathcal{I}|C)^\star$-universal sets. 

If there is a $\beta<\alpha$ such that $D=C\cap f^{-1}[A_\beta]\notin\I$ then put $A_\alpha=A_\beta$ and note that $f_\alpha|D$ is finite-to-one and $f_\alpha[D]\subseteq A_\alpha$. 

Otherwise, $C\cap f^{-1}[A_\beta]\in\I$ for all $\beta<\alpha$. Consider the following poset: conditions of $\mathbb{P}$ are pairs $p=(B_p,F_p)\in[C]^{<\omega}\times[\alpha]^{<\omega}$ and for $p,q\in\mathbb{P}$ let $q\vdash p$ if and only if $B_p\subseteq B_q$, $F_p\subseteq F_q$ and $B_q\setminus B_p\subseteq C\setminus f_\alpha^{-1}[\bigcup_{\beta\in F_p}A_\beta]$.

Note that $\mathbb{P}$ is $\sigma$-centered as $\mathbb{P}=\bigcup_{B\in[C]^{<\omega}}\{p\in\mathbb{P}:\ B_p=B\}$ and given any $B\in[C]^{<\omega}$ and any $p$ and $q$ with $B_p=B_q=B$ we have $(B,F_p\cup F_q)\vdash p$ and $(B,F_p\cup F_q)\vdash q$.

Observe that for each $\beta<\alpha$ the open set $D_\beta=\{p\in\mathbb{P}:\ \beta\in F_p\}$ is dense in $\mathbb{P}$. Indeed, given any $p\in\mathbb{P}$ put $q=(B_p,F_p\cup\{\beta\})$ and note that $q\in D_\beta$ and $q\vdash p$.

Let $\{\cZ_k:k\in\omega\}$ be the family witnessing $\omega$-diagonalizability of $\I|C$ by $(\mathcal{I}|C)^\star$-universal sets. Recall that:
\begin{itemize} 
\item for each $k\in\omega$ the family $\cZ_k\subseteq [C]^{<\omega}\setminus \{\emptyset\}$ is such that for each $D\in \I|C$ there is a $Z\in \cZ_k$ with $Z\cap D=\emptyset$;
\item for each $D\in \I|C$ there is a $k\in\omega$ such that $Z\not\subseteq D$ for every $Z\in \cZ_k$. 
\end{itemize}
Define open sets $E_k=\{p\in\mathbb{P}:\ (\exists Z\in \cZ_k)\ Z\subseteq B_p\}$ for all $k\in\omega$. We need to show that each $E_k$ is dense. Let $k\in\omega$ and $p\in\mathbb{P}$. Since $C\cap f^{-1}[A_\beta]\in\I$ for all $\beta<\alpha$, we have $f_\alpha^{-1}[\bigcup_{\beta\in F_p}A_\beta]\in\I|C$. Thus, there is a $Z\in \cZ_k$ with $Z\cap f_\alpha^{-1}[\bigcup_{\beta\in F_p}A_\beta]=\emptyset$. Put $q=(B_p\cup Z,F_p)$. Clearly, $q\in E_k$ and $q\vdash p$. Thus, $E_k$ is dense.

Let $\mathcal{D}=\{E_n:\ n\in\omega\}\cup\{D_\beta:\ \beta<\alpha\}$. Then $|\mathcal{D}|<\cc$ and $\mathcal{D}$ consists of open dense sets. By MA($\sigma$-centered), there is a $\mathcal{D}$-generic filter $G\subseteq\mathbb{P}$. Define $B=\bigcup\{B_p:\ p\in G\}$. 

Observe that $B\notin\I|C$. Indeed, $B$ contains some $Z\in \cZ_k$, for each $k\in\omega$ (since $G\cap E_k\neq\emptyset$). However, $B\in\I|C$ would imply that there is a $k\in\omega$ such that $Z\not\subseteq B$ for every $Z\in \cZ_k$.

Note also that $f_\alpha[B]\cap A_\beta$ is finite for all $\beta<\alpha$. Indeed, fix $\beta<\alpha$ and observe that there is a $p\in G$ with $F_p\ni\beta$ (since $G\cap D_\beta\neq\emptyset$). We claim that $f_\alpha[B]\cap A_\beta\subseteq B_p\in[C]^{<\omega}$. To show this inclusion, fix any $q\in G$. Since $G$ is a filter, we can find $q'\in G$ with $q'\vdash q$ and $q'\vdash p$. Then $B_q\setminus B_p\subseteq B_{q'}\setminus B_p\subseteq C\setminus f_\alpha^{-1}[\bigcup_{\delta\in F_p}A_\delta]\subseteq C\setminus f_\alpha^{-1}[A_\beta]$.

Define $A_\alpha=f_\alpha[B]$ and note that $f_\alpha|B$ is finite-to-one (since $B\subseteq C$). This finishes the induction step and it is clear that the family $\mathcal{A}=\{A_\alpha:\ \alpha<\cc\}$ has the required properties.
\end{proof}

\section{Main results}

\begin{proposition}
\label{omega_1}
If $\I$ is a hereditary weak P-ideal then $\omega_1$ with order topology is in FinBW($\I$). In particular, there is a non-compact space in FinBW($\I$). 
\end{proposition}

\begin{proof}
Let $g:\omega\to\omega_1$. Without loss of generality we can assume that $g^{-1}[\{\alpha\}]\in\I$ for all $\alpha<\omega_1$. 

Find $x<\omega_1$ such that $B=g^{-1}[[0,x)]\notin\I$ and $g^{-1}[[0,x')]\in\I$ for all $x'<x$. 

Let $(\alpha_n)\subseteq $x be any increasing sequence such that $\alpha_0=0$ and $\lim_n \alpha_n=$x. Fix any injection $h:B\to\omega^2$ such that $h[g^{-1}[[\alpha_n,\alpha_{n+1})]]\subseteq\{n\}\times\omega$, for all $n\in\omega$. Observe that $h^{-1}[\{n\}\times\omega]\subseteq g^{-1}[[0,\alpha_{n+1})]\in\I|B$ for all $n$, as $\alpha_{n+1}<x$. Since $\I$ is a hereditary weak P-ideal, $\fin^2\not\sqsubseteq\I|B$. Thus, there is a $C\notin\I$, $C\subseteq B$ such that $g[C]\cap[\alpha_n,\alpha_{n+1})$ is finite for all $n$. This means that $(g(n))_{n\in C}$ converges to $x$.
\end{proof}

\begin{proposition}
\label{omega^2}
If $\I$ is unboring then $\omega^2+1$ with order topology is in FinBW($\I$).
\end{proposition}

\begin{proof}
Let $g:\omega\to\omega^2+1$. Without loss of generality we can assume that $g^{-1}[\{\alpha\}]\in\I$ for all $\alpha<\omega^2$. Define $X_{(0,0)}=g^{-1}[\{0\}]\cup g^{-1}[\{\omega^2\}]$ and $X_{(i,j)}=g^{-1}[\{i\cdot\omega+j\}]$ for all $(i,j)\in\omega^2\setminus\{(0,0)\}$. Then each $X_{(i,j)}$ belongs to $\I$. By Proposition \ref{q}(iv), there is an $A\notin\I$ such that $A\cap X_{(i,j)}$ is finite for all $i,j\in\omega$ and $A\cap\bigcup_{j}X_{(i,j)}$ is finite for almost all $i\in\omega$, i.e., $\{(i,j)\in\omega:\ A\cap X_{(i,j)}\neq\emptyset\}\in\fin^2$. If $A_i=A\cap\bigcup_{j}X_{(i,j)}\notin\I$ for some $i\in\omega$ then $(i+1)\cdot\omega$ is a limit of $(g(n))_{n\in A_i}$. Otherwise, if $A_i\in\I$ for all $i\in\omega$ then $\omega^2$ is a limit of $(g(n))_{n\in A'}$, where $A'=\bigcup\{A_i:\ A_i\text{ is finite}\}$.
\end{proof}

\begin{proposition}
\label{f}
If $\I$ is boring then each space in FinBW($\I$) is boring. 
\end{proposition}

\begin{proof}
We will apply Proposition \ref{q}(v). Suppose that $\I$ is boring and $f:\omega\to\omega^2$ is the witnessing function. Observe that $f^{-1}[\{(i,j)\}]\in\I$, for all $(i,j)\in\omega^2$ (otherwise for $A=f^{-1}[\{(i,j)\}]\notin\I$ the function $f|A$ would be constant and $f[A]$ would belong to $\fin^2$).

Fix any unboring $X$. Then one can find an infinite set $\{x_i:\ i\in\omega\}\subseteq X$ and a family of injective convergent sequences $\{(x_{i,j})_j:\ i\in\omega\}$ such that $\lim_j x_{i,j}=x_i$ for each $i\in\omega$. Without loss of generality we can assume that $x_{i,j}\neq x_{i',j'}$ whenever $(i,j)\neq(i',j')$ (since the intersection $\{x_{i,j}:\ j\in\omega\}\cap\{x_{i',j}:\ j\in\omega\}$ is finite, for all $i\neq i'$, as $(x_{i,j})_j$ and $(x_{i',j})_j$ have different limits).

Define a sequence $g:\omega\to X$ by $g(n)=x_{f(n)}$. We claim that for each $A\subseteq\omega$, if $(g(n))_{n\in A}$ is convergent, then $A\in\I$ (i.e., that $X$ is not in FinBW($\I$)).

Let $A\subseteq\omega$ be such that $(g(n))_{n\in A}$ is convergent. Since $f^{-1}[\{(i,j)\}]\in\I$, for all $(i,j)\in\omega^2$, if $g[A]$ is finite, then $A\in\I$. Thus, as $(g(n))_{n\in A}$ is convergent, we can assume that $g|A$ is finite-to-one. Then also $f|A$ is finite-to-one. Moreover, if $(g(n))_{n\in A}$ converges, then $f[A]\in\fin^2$. Consequently, $A\in\I$. 
\end{proof}

\begin{corollary}
\label{c}
If $\I$ is boring then each space in FinBW($\I$) is compact and there are no uncountable separable spaces in FinBW($\I$).
\end{corollary}

\begin{proof}
If $\I$ is boring then each space in FinBW($\I$) is boring by Proposition \ref{f}. Thus, by Proposition \ref{boringspaces}, each space in FinBW($\I$) is compact and there are no uncountable separable spaces in FinBW($\I$).
\end{proof}

Recall that for any ideal $\I$, each finite space is in FinBW($\I$) and each space from FinBW($\I$) is sequentially compact.

\begin{theorem}
\label{extreme}
The following hold for any ideal $\I$:
\begin{enumerate}
\item[(a)] $\fin^2\sqsubseteq\I$ if and only if FinBW($\I$) coincides with finite spaces.
\item[(b)] $\I$ is boring and $\fin^2\not\sqsubseteq\I$ if and only if FinBW($\I$) coincides with boring spaces.
\item[(c)] $\I$ is not tall if and only if FinBW($\I$) coincides with sequentially compact spaces.
\end{enumerate} 
\end{theorem}

\begin{proof}
Without loss of generality we can assume that $\I$ is an ideal on $\omega$.

(a): Suppose first that $\fin^2\sqsubseteq\I$ and let $f:\omega\to\omega^2$ be a bijection witnessing it. Suppose that there is an infinite Hausdorff $X\in\text{FinBW}(\I)$ and let $(x_n)\subseteq X$ be an injective sequence. Define $y_{f^{-1}(i,j)}=x_{i}$ for all $i,j\in\omega$ (note that the sequence $(y_n)$ is well-defined as $f$ is a bijection). Suppose that $(y_n)_{n\in A}$ converges to $z\in X$, for some $A\subseteq\omega$. Then $A\cap f^{-1}[\{i\}\times\omega]$ is finite whenever $x_i\neq z$ (because there is an open set $U\ni z$ with $x_i\notin U$, since $X$ is Hausdorff). Thus, either $z\notin\{x_i:i\in\omega\}$ and $A\cap f^{-1}[\{i\}\times\omega]$ is finite for all $i\in\omega$ or $z=x_n$ for some $n\in\omega$ and $A\cap f^{-1}[\{i\}\times\omega]$ is finite for all $i\in\omega\setminus\{n\}$. In both cases $A\in\I$. Therefore, $X$ cannot be in FinBW($\I$).

Suppose now that $\fin^2\not\sqsubseteq\I$. Consider the space $X=\{0\}\cup\{\frac{1}{n+1}:\ n\in\omega\}$ with the subspace topology inherited from $\mathbb{R}$. We claim that $X$ is an infinite space belonging to FinBW($\I$). Let $f:\omega\to X$. Without loss of generality we may assume that $f^{-1}[\{x\}]\in \I$ for all $x\in X$. The family $(f^{-1}[\{x\}])_{x\in X}$ defines a partition of $\omega$ into sets belonging to $\I$. As $\fin^2\not\sqsubseteq\I$, we can find $A\notin\I$ such that $A\cap f^{-1}[\{x\}]$ is finite for all $x\in X$. Thus, $(f(n))_{n\in A}$ is a finite-to-one subsequence. In the space $X$ each such subsequence converges to $0\in X$. Therefore, $X$ is in FinBW($\I$).

(b): If $\I$ is unboring then, by Proposition \ref{omega^2}, there is an unboring Hausdorff space in FinBW($\I$). If $\fin^2\sqsubseteq\I$ then FinBW($\I$) coincides with finite spaces by item (a), so not every boring space is in FinBW($\I$) (by Proposition \ref{boringspaces}(i)). Moreover, by Proposition \ref{boringspaces2}, if $\fin^2\not\sqsubseteq\I$ then each boring space is in FinBW($\I$). Finally, if $\I$ is boring then all Hausdorff spaces in FinBW($\I$) are boring by Proposition \ref{f}.

(c): This is \cite[Proposition 2.4]{FT}.
\end{proof}

\begin{theorem}
\label{char}
The following are equivalent for any ideal $\I$:
\begin{itemize}
\item[(a)] $\I$ is unboring;
\item[(b)] there is a Hausdorff unboring space in FinBW($\I$).
\end{itemize}
If additionally CH holds, then the above are equivalent to:
\begin{itemize}
\item[(c)] there is an uncountable Mr\'{o}wka space in FinBW($\I$);
\item[(d)] there is an uncountable separable space in FinBW($\I$).
\end{itemize}
\end{theorem}

\begin{proof}
(a)$\implies$(b): This is Proposition \ref{omega^2}. 

(b)$\implies$(a): This is Proposition \ref{f}.

(a)$\implies$(c): Under CH, this is Theorem \ref{Mrowka}. 

(c)$\implies$(d): Every Mr\'{o}wka space is separable.

(d)$\implies$(a): This is Corollary \ref{c}. 
\end{proof}

Since the implication (d)$\implies$(a) in Theorem \ref{char} is true in ZFC, we are able to state the following result.

\begin{corollary}
\label{absolute}
Let $\I$ be an ideal.
\begin{itemize}
\item[(a)] If $\I$ is Borel, then the following are equivalent:
\begin{itemize}
\item[(i)] $\I$ is unboring in some forcing extension;
\item[(ii)] $\I$ is unboring in every forcing extension;
\item[(iii)] in every forcing extension in which CH holds, there is an uncountable separable space in FinBW($\I$);
\item[(iv)] there is a forcing extension in which there is an uncountable separable space in FinBW($\I$).
\end{itemize}
\item[(b)] If either $\I$ being unboring holds in ZFC or $\I$ being boring holds in ZFC, then the following are equivalent:
\begin{itemize}
\item[(i)] $\I$ is unboring;
\item[(ii)] under CH there is an uncountable separable space in FinBW($\I$);
\item[(iii)] it is consistent that there is an uncountable separable space in FinBW($\I$).
\end{itemize}
\end{itemize}
\end{corollary}

\begin{proof}
(a): Since $\I$ and $\BI$ are Borel, $\I$ being unboring is an absolute statement (by \cite[Proposition 1.3]{Hrusak}). Thus, (i)$\implies$(ii) by Shoenfield’s absoluteness theorem, (ii)$\implies$(iii) follows from Theorem \ref{char}, (iii)$\implies$(iv) is trivial and (iv)$\implies$(i) follows from Corollary \ref{c}. 

(b): The implication (i)$\implies$(ii) is true by Theorem \ref{char}, (ii)$\implies$(iii) is trivial and if (iii) holds then it is consistent that $\I$ is unboring (by Corollary \ref{c}), so using our assumption we get item (i).
\end{proof}

\begin{corollary}
\label{corollary-main}
Suppose that $\I$ is extendable to a coanalytic weak P-ideal (in particular, to a $\bf{\Pi^0_4}$ ideal). Then $\omega_1$ with the order topology is in FinBW($\I$). In particular, there is a non-compact space in FinBW($\I$). Moreover, if MA($\sigma$-centered) holds, then there is an uncountable separable Mr\'{o}wka space in FinBW($\I$).
\end{corollary}

\begin{proof}
Follows from Propositions \ref{all-cor}, \ref{r}(b), \ref{omega_1} and Theorem \ref{MAlem}.
\end{proof}

\begin{corollary}
\label{m1}
Let $\I$ be a $\leq_K$-uniform ideal. Then the following are equivalent:
\begin{itemize}
\item[(a)] $\I$ is a weak P-ideal;
\item[(b)] there is a Hausdorff unboring space in FinBW($\I$);
\item[(c)] $\omega_1$ with order topology is in FinBW($\I$);
\item[(d)] there is a non-compact space in FinBW($\I$).
\end{itemize}
If additionally CH holds, then the above are equivalent to:
\begin{itemize}
\item[(e)] there is an uncountable Mr\'{o}wka space in FinBW($\I$);
\item[(f)] there is an uncountable separable space in FinBW($\I$).
\end{itemize}
If $\I$ is coanalytic then in the above we can replace CH with MA($\sigma$-centered).
\end{corollary}

\begin{proof}
As $\I$ is $\leq_K$-uniform, $\I$ is a weak P-ideal if and only if $\I$ is unboring (by Corollary \ref{K-uniform}). This together with Theorem \ref{char} shows the equivalence of (a), (b), (e) and (f) (under CH). Recall that each unboring $\leq_K$-uniform ideal is a hereditary weak P-ideal (again by Corollary \ref{K-uniform}). Thus, Proposition \ref{omega_1} gives us the implication (a)$\implies$(c). The implication (c)$\implies$(d) is obvious and (d)$\implies$(a) follows from Corollary \ref{c}.

If $\I$ is coanalytic, then (a)$\implies$(e) under MA($\sigma$-centered) follows from Corollaries \ref{K-uniform} and \ref{corollary-main}. Finally, the implication (e)$\implies$(f) is obvious (as every Mr\'{o}wka space is separable) and (f)$\implies$(a) is proved in Corollary \ref{c}.
\end{proof}

\begin{corollary}
\label{maximal}
Suppose that $\cU$ is an ultrafilter.
\begin{itemize}
\item If $\cU$ is not a P-point then FinBW($\cU^\star$) coincides with finite spaces.
\item If $\cU$ is a P-point then the $[0,1]$ interval and $\omega_1$ with the order topology are in FinBW($\cU^\star$). In particular, in FinBW($\cU^\star$) there is an uncountable separable space as well as a non-compact space.
\end{itemize}
\end{corollary}

\begin{proof}
By Proposition \ref{P-point}, the following are equivalent:
\begin{itemize}
\item $\cU$ is a P-point;
\item the $[0,1]$ interval is in FinBW($\cU^\star$);
\item $\cU^\star$ is a weak P-ideal
\end{itemize}
Thus, the first statement is a consequence of Theorem \ref{extreme}(a). For the second statement, recall that each maximal ideal is $\leq_K$-uniform (see Example \ref{uniform-ex}). Thus, the thesis follows from Theorem \ref{extreme} and Corollary \ref{m1}. 
\end{proof}

\section{Some (open) problems}

We start this section with the following question posed by D. Meza-Alc\'antara in \cite[Question 4.4.7]{Meza}: Is it true that, if $\I$ is a Borel ideal then either $\fin^2\leq_K\I$ or there is a $\bf{\Pi^0_3}$ ideal containing $\I$? This problem has been repeated by M. Hru\v{s}\'ak in \cite[Question 5.18]{Hrusak} and recently in \cite[Question 5.10]{Hrusak2}. In the next example we answer it in the negative. 

\begin{example}
Consider the ideal $\BI$. It is Borel (by Remark \ref{rem-Borel}), $\fin^2\not\leq_K\BI$ (by Propositions \ref{propertyKat} and \ref{conv}), but $\BI$ cannot be extended to any $\bf{\Pi^0_3}$ ideal. In fact, $\BI$ even cannot be extended to any $\bf{\Pi^0_4}$ ideal, as $\BI$ is boring and each ideal extendable to a $\bf{\Pi^0_4}$ ideal cannot be boring (by Propositions \ref{all-cor} and \ref{boringprop}(c)).
\end{example}

The above example shows also that $\fin^2$ is not critical (in the sense of $\leq_K$) for extendability to ideals of any Borel class (among Borel ideals): if there is a critical ideal for extendability to $\bf{\Pi^0_4}$ ideals, then it must be $\leq_K$-below $\BI$ (by the above example) and $\fin^2$ cannot be critical for extendability to $\bf{\Sigma^0_4}$ ideals as $\fin^2$ is itself $\bf{\Sigma^0_4}$.

However, there is still a chance that extendability to some Borel class is related to containing an isomorphic copy of $\fin^2$:

\begin{problem}
\label{question}
Is every Borel hereditary weak P-ideal extendable to a $\bf{\Pi^0_4}$ ideal?
\end{problem}

Note that the converse implication is false: $\I=\fin\oplus\fin^2$ is a Borel ideal extendable to a $\bf{\Sigma^0_2}$ ideal $\{A\subseteq(\{0\}\times\omega)\cup(\{1\}\times\omega^2):\ A\cap(\{0\}\times\omega)\text{ is finite}\}$, but it is not a hereditary weak P-ideal (as $\I|(\{1\}\times\omega^2)\cong\fin^2$).

An equivalent formulation of the above question is the following: is it true that for each Borel ideal $\I$ the following are equivalent:
\begin{itemize}
\item[(a)] $\I$ is extendable to a hereditary weak P-ideal;
\item[(b)] $\I$ is extendable to a $\bf{\Pi^0_4}$ ideal?
\end{itemize}

Observe that the implication (b)$\implies$(a) is true by Proposition \ref{all-cor}. 

Next example shows that in Question \ref{question} we cannot omit the assumption that the ideal is Borel. 

\begin{example}
If there are P-points (in particular, if MA($\sigma$-centered) holds) then there are hereditary weak P-ideals which are not extendable to a $\bf{\Pi^0_4}$ ideal. Indeed, let $\mathcal{U}$ be a P-point and denote $\I=\mathcal{U}^\star$. By Proposition \ref{P-point}, $\fin^2\not\sqsubseteq\I$. Since $\I$ is a maximal ideal, it is homogeneous (see Example \ref{uniform-ex}). Thus, $\I$ is a hereditary weak P-ideal. On the other hand, since $\I$ is maximal, the only ideal containing $\I$ is $\I$ itself. As no maximal ideal is Borel, $\I$ cannot be extended to a $\bf{\Pi^0_4}$ ideal.
\end{example}

By Proposition \ref{r}(d), under MA($\sigma$-centered) there are unboring ideals which are not strongly unboring. However, a positive answer to the following problem would mean that in Theorem \ref{char} we can replace CH with MA($\sigma$-centered) in the case of all Borel ideals (by Theorem \ref{MAlem}).

\begin{problem}
Is every Borel unboring ideal strongly unboring?
\end{problem}

We end this section with some comments about P-points under MA($\sigma$-centered). By Corollary \ref{maximal}, for each P-point $\cU$ there is an uncountable separable space in FinBW($\cU^\star$) (namely, the $[0,1]$ interval). However, by Proposition \ref{r}(d), if $\cU$ is a P-point then $\cU^\star$ is not strongly unboring. Thus, Theorem \ref{MAlem} does not give us an uncountable Mr\'{o}wka space in FinBW($\cU^\star$). 

\begin{problem}
Assume MA($\sigma$-centered). Is there a P-point $\cU$ such that there is no uncountable Mr\'{o}wka space in FinBW($\cU^\star$)?
\end{problem}

\section{New preorder on ideals}

\begin{definition}
If $\I$ is an ideal and $(G_{i,j})\subseteq\I$ is a partition of $\bigcup\I$ then $\hat{\I}(G_{i,j})$ is an ideal on $\omega^2$ consisting of all $A\subseteq\omega^2$ with:
$$\left(\forall X\subseteq\bigcup_{(i,j)\in A}G_{i,j}\right)\ \left(\left(\left(\forall (i,j)\in A\right)\ X\cap G_{i,j}\text{ is finite}\right)\implies X\in\I\right).$$
\end{definition}

\begin{definition}
Let $\I$ and $\J$ be ideals. We write $\J\preceq\I$ if for each partition $(H_{i,j})\subseteq\J$ with $\fin\otimes\emptyset\subseteq\hat{\J}(H_{i,j})$ there is a partition $(G_{i,j})\subseteq\I$ such that $\hat{\J}(H_{i,j})\cap\fin^2\subseteq\hat{\I}(G_{i,j})$.
\end{definition}

One could think that in the above definition the condition $\fin\otimes\emptyset\subseteq\hat{\J}(H_{i,j})$ is unnecessary or that the strange condition $\hat{\J}(H_{i,j})\cap\fin^2\subseteq\hat{\I}(G_{i,j})$ could be replaced by a simpler one: $\hat{\J}(H_{i,j})\subseteq\hat{\I}(G_{i,j})$. As we will see in Proposition \ref{prec_2}, the relation $\preceq$ is designed to have the following property: $\J\preceq\I$ implies $\text{FinBW}(\I)\subseteq\text{FinBW}(\J)$, whenever $\J$ is a weak P-ideal. The mentioned above two modifications of $\preceq$ would also have the required property, but $\preceq$ is the weakest one among them and we wanted this new relation to be as accurate for this purpose as it can be.

\begin{proposition}
The relation $\preceq$ is a preorder on the set of ideals.
\end{proposition}

\begin{proof}
It is clear that $\preceq$ is reflexive, so we only show transitivity. Suppose that $\I\preceq\I'$ and $\I'\preceq\I''$. Fix any partition $(H_{i,j})\subseteq\I$ with $\fin\otimes\emptyset\subseteq\hat{\I}(H_{i,j})$. Then there is a partition $(H'_{i,j})\subseteq\I'$ such that $\hat{\I}(H_{i,j})\cap\fin^2\subseteq\hat{\I'}(H'_{i,j})$. Since $\fin\otimes\emptyset\subseteq\hat{\I}(H_{i,j})\cap\fin^2\subseteq\hat{\I'}(H'_{i,j})$, there is a partition $(H''_{i,j})\subseteq\I''$ such that $\hat{\I'}(H'_{i,j})\cap\fin^2\subseteq\hat{\I''}(H''_{i,j})$. Then we have $\hat{\I}(H_{i,j})\cap\fin^2\subseteq\hat{\I'}(H'_{i,j})\cap\fin^2\subseteq\hat{\I''}(H''_{i,j})$.
\end{proof}

\begin{proposition}
\label{p1}
Let $\I$ and $\J$ be ideals.
\begin{itemize}
\item[(a)] The following are equivalent:
\begin{itemize}
\item[(i)] $\I$ is boring;
\item[(ii)] $\fin^2\preceq\I$;
\item[(iii)] $\J\preceq\I$ for all ideals $\J$.
\end{itemize}
\item[(b)] If $\J\leq_K\I$ then $\J\preceq\I$.
\item[(c)] Suppose that $\I$ is unboring. If $\J\preceq\I$ then $\J\leq_K\I|A$ for some $A\notin\I$.
\end{itemize}
\end{proposition}

\begin{proof}
Without loss of generality we may assume that $\I$ and $\J$ are ideals on $\omega$.

(a): The implication (iii)$\implies$(ii) is obvious. If $\fin^2\preceq\I$ then for the partition given by $H_{i,j}=\{(i,j)\}\in\fin^2$, there is a partition $(G_{i,j})\subseteq\I$ such that $\fin^2\subseteq\hat{\I}(G_{i,j})$. Observe that this implies $\I$ being boring, by Proposition \ref{q}(iv). Indeed, $(G_{i,j})\subseteq\I$ is a partition such that given any $X\subseteq\omega$, if $X\cap G_{i,j}\in\fin$ for all $(i,j)\in\omega^2$ and $X\cap \bigcup_{j}G_{i,j}\in\fin$ for almost all $i\in\omega$, then $A=\{(i,j)\in\omega:\ X\cap G_{i,j}\neq\emptyset\}\in\fin^2$, so also $A\in\hat{\I}(G_{i,j})$ and, consequently, $X\in\I$ (as $X\cap G_{i,j}\in\fin$ for all $(i,j)\in\omega^2$). 

On the other hand, if $\I$ is boring then, by Proposition \ref{q}(iv), there is a partition $(G_{i,j})\subseteq\I$ such that for any $X\subseteq\omega$, if $X\cap G_{i,j}\in\fin$ for all $(i,j)\in\omega^2$ and $X\cap \bigcup_{j}G_{i,j}\in\fin$ for almost all $i\in\omega$, then $X\in\I$. Fix any $A\in\fin^2$. Then for every $X\subseteq\omega$ such that $X\subseteq\bigcup_{(i,j)\in A}G_{i,j}$ and $X\cap G_{i,j}\in\fin$ for all $(i,j)\in\omega^2$ we get $X\in\I$. Thus, $A\in \hat{\I}(G_{i,j})$. Thus, $\fin^2\subseteq\hat{\I}(G_{i,j})$ and $\J\preceq\I$ for every ideal $\J$.

(b): Let $f:\omega\to\omega$ witness $\J\leq_K\I$ and fix any partition $(H_{i,j})\subseteq\J$ of $\omega$ with $\fin\otimes\emptyset\subseteq\hat{\J}(H_{i,j})$. Define $G_{i,j}=f^{-1}[H_{i,j}]$ for all $(i,j)\in\omega^2$. Then $(G_{i,j})\subseteq\I$ is a partition of $\omega$. Let $A\in\hat{\J}(H_{i,j})\cap\fin^2$. We need to prove that $A\in\hat{\I}(G_{i,j})$. Let $X\subseteq\bigcup_{(i,j)\in A}G_{i,j}$ be such that $X\cap G_{i,j}\in\fin$ for all $(i,j)\in A$. Then $f[X]\subseteq\bigcup_{(i,j)\in A}H_{i,j}$ and $f[X]\cap H_{i,j}\in\fin$ for all $(i,j)\in A$. Since $A\in\hat{\J}(H_{i,j})$, $f[X]$ is in $\J$. Therefore, $X\subseteq f^{-1}[f[X]]\in\I$ and hence $A\in\hat{\I}(G_{i,j})$. 

(c): Suppose that $\I$ is unboring and $\J\not\leq_K\I|A$ for all $A\notin\I$. We will show that $\J\not\preceq\I$. Define $H_{i,0}=\{i\}$ and $H_{i,j}=\emptyset$ for all $i\in\omega$ and $j\in\omega\setminus\{0\}$. Clearly, $\fin\otimes\emptyset\subseteq\hat{\J}(H_{i,j})$. Fix any partition $(G_{i,j})\subseteq\I$. Since $\I$ is unboring, by Proposition \ref{q}(iv) there is an $X\notin\I$ with $X\cap G_{i,j}\in\fin$ for all $(i,j)\in\omega^2$ and $B=\{(i,j)\in\omega^2:\ X\cap G_{i,j}\neq\emptyset\}\in\fin^2$. Define $f:X\to\omega$ by 
$$f(x)=i\ \Longleftrightarrow\ x\in G_{i,j}$$
for all $x\in X$. As $\J\not\leq_K\I|X$, there is a $Y\in\J$ with $f^{-1}[Y]\notin\I$. Define $A=\{(i,j)\in \omega^2:\ i\in Y\}\cap B$. Note that $A\in\hat{\J}(H_{i,j})\cap\fin^2$ as $Y\in\J$ and $B\in\fin^2$. However, $A\notin\hat{\I}(G_{i,j})$, since $\I\not\ni f^{-1}[Y]\subseteq\bigcup_{(i,j)\in A} G_{i,j}$ and $f^{-1}[Y]\cap G_{i,j}\subseteq X\cap G_{i,j}\in\fin$ for all $(i,j)\in A$. 
\end{proof}

\begin{proposition}
\label{1<->2}
If $\I$ is a hereditary weak P-ideal, then $\J\preceq\I$ and $\J\leq_K\I$ are equivalent, for any ideal $\J$.
\end{proposition}

\begin{proof}
By Proposition \ref{p1}(b), $\J\leq_K\I$ implies $\J\preceq\I$.

Assume now that $\I$ is a hereditary weak P-ideal such that $\J\preceq\I$. Without loss of generality we may assume that $\I$ and $\J$ are ideals on $\omega$. Define $H_{i,0}=\{i\}$ and $H_{i,j}=\emptyset$ for all $i\in\omega$ and $j\in\omega\setminus\{0\}$. Obviously, $\fin\otimes\emptyset\subseteq\hat{\J}(H_{i,j})$. Since $\J\preceq\I$, there is a partition $(G_{i,j})\subseteq\I$ such that $\hat{\J}(H_{i,j})\cap\fin^2\subseteq\hat{\I}(G_{i,j})$. 

Define $f:\omega\to\omega$ by
$$f(x)=i\ \Longleftrightarrow\ x\in G_{i,j}$$
for all $x\in \omega$. We claim that $f$ witnesses $\J\leq_K\I$. Indeed, fix any $A\in\J$. Notice that $B=f^{-1}[A]=\bigcup\{G_{i,j}:\ i\in A\}$. Suppose to the contrary that $B\notin\I$. Since $\I$ is a hereditary weak P-ideal, $\I|B$ is unboring (by Proposition \ref{hh}), so there is a $C\notin\I|B$, $C\subseteq B$ such that $C\cap G_{i,j}\in\fin$, for all $(i,j)\in\omega^2$, and $D=\{(i,j)\in\omega^2:\ C\cap G_{i,j}\neq\emptyset\}\in\fin^2$ (by Proposition \ref{q}(iv)). Observe that $D\notin \hat{\I}(G_{i,j})$ (as $C\notin\I$). However,  
$$\{i\in\omega:\ (i,j)\in D\text{ for some }j\in\omega\}=\{i\in\omega:\ C\cap G_{i,j}\neq\emptyset\text{ for some }j\in\omega\}$$
$$\subseteq\{i\in\omega:\ B\cap G_{i,j}\neq\emptyset\text{ for some }j\in\omega\}=A\in\J,$$
since $C\subseteq B$. Therefore, $D\in\hat{\J}(H_{i,j})$, which contradicts the choice of $(G_{i,j})$ and finishes the proof.
\end{proof}

A natural question is whether $\leq_K$ and $\prec$ coincide for all ideals. The remaining part of this section is devoted to that problem. It occurs that the answer is negative even for $\bf{\Sigma^0_4}$ ideals. We need to prove three lemmas before providing a suitable example.

\begin{lemma}
\label{l1}
Let $\J$ be a hereditary weak P-ideal on $\omega$. Then the ideal $\fin^2\cap(\J\otimes\emptyset)$ is unboring.
\end{lemma}

\begin{proof}
Denote $\I=\fin^2\cap(\J\otimes\emptyset)$ and let $\pi_1:\omega^2\to\omega$ be the projection onto the first coordinate, i.e., $\pi_1(i,j)=i$ for all $(i,j)\in\omega^2$. Fix any bijection $f:\omega^2\to\omega^3$. We need to find $A\in\BI$ such that $f^{-1}[A]\notin\I$. There are three possible cases:

If $f^{-1}[\{(i,j)\}\times\omega]\notin\I$ for some $(i,j)\in\omega^2$, then $\{(i,j)\}\times\omega\in\mathcal{BI}$ is the required set.

If $f^{-1}[\{(i,j)\}\times\omega]\in\I$ for all $(i,j)\in\omega^2$, but $A=\pi_1[f^{-1}[\{i\}\times\omega^2]]\notin\J$ for some $i\in\omega$, then consider the family $\{\pi_1[f^{-1}[\{(i,j)\}\times\omega]]:\ j\in\omega\}\subseteq\J|A$ (note that it covers $A$). Since $\J$ is a hereditary weak P-ideal, there is a $B\subseteq A$, $B\notin\J$ such that $B\cap\pi_1[f^{-1}[\{(i,j)\}\times\omega]]$ is finite for all $j\in\omega$. For each $j\in\omega$ find a finite set $F_j\subseteq f^{-1}[\{(i,j)\}\times\omega]$ with $\pi_1[F_j]=B\cap\pi_1[f^{-1}[\{(i,j)\}\times\omega]]$. Define $C=f[\bigcup_j F_j]$. Then $f^{-1}[C]=\bigcup_j F_j\notin\I$ (as $\pi_1[\bigcup_j F_j]=B\notin\J$). On the other hand, $C\in\mathcal{BI}$ as $C\subseteq \{i\}\times\omega^2$ and $C\cap(\{(i,j)\}\times\omega)=f[F_j]$ is finite for all $j\in\omega$.

If $\pi_1[f^{-1}[\{i\}\times\omega^2]]\in\J$ for all $i\in\omega$, then the family $\{\pi_1[f^{-1}[\{i\}\times\omega^2]]:\ i\in\omega\}\subseteq\J$ and, as $\J$ is a weak P-ideal, we can find $B\notin\J$ such that $B\cap\pi_1[f^{-1}[\{i\}\times\omega^2]]$ is finite for all $i\in\omega$. Similarly as in the previous paragraph, for each $i\in\omega$ find a finite set $F_i\subseteq f^{-1}[\{i\}\times\omega^2]$ with $\pi_1[F_i]=B\cap\pi_1[f^{-1}[\{i\}\times\omega^2]]$. Then for $C=f[\bigcup_i F_i]$ we have $f^{-1}[C]\notin\I$, but $C\in\mathcal{BI}$ as $C\cap(\{i\}\times\omega^2)=f[F_i]$ is finite for all $i\in\omega$.
\end{proof}

\begin{lemma}
\label{l2}
Let $\J$ be a tall P-ideal on $\omega$. Then $\J\not\leq_K\fin^2\cap(\J\otimes\emptyset)$.
\end{lemma}

\begin{proof}
Denote $\I=\fin^2\cap(\J\otimes\emptyset)$ and fix any $f:\omega^2\to\omega$. If there are infinitely many $i\in\omega$ with $f[\{i\}\times\omega]$ finite, then denote $T=\{n\in\omega:\ (\exists i\in\omega)\ f(i,j)=n\text{ for infinitely many }j\in\omega\}$. If $T\in\fin\subseteq\J$ then $f^{-1}[T]\notin\I$ (as there are infinitely many $i\in\omega$ with $f[\{i\}\times\omega]\in\fin$) and we are done. On the other hand, if $T$ is infinite then, using the tallness of $\J$, find an infinite $A\in\J$ with $A\subseteq T$. Then $f^{-1}[A]\notin\I$ as $f^{-1}[A]\notin\fin^2$. 

Assume now that $f[\{i\}\times\omega]$ is infinite for almost all $i\in\omega$ and denote by $T$ the set of all those $i$s. Since $\J$ is tall, for each $i\in T$ we can find an infinite $A_i\in\J$ with $A_i\subseteq f[\{i\}\times\omega]$. Using the fact that $\J$ is a P-ideal, find $A\in\J$ with $A_i\setminus A$ finite for all $i\in T$. Then $f^{-1}[A]\notin\I$ as $f^{-1}[A]\cap (\{i\}\times\omega)$ is infinite for all $i\in T$.
\end{proof}

\begin{lemma}
\label{l3}
If $\J$ is a $\leq_{KB}$-uniform weak P-ideal on $\omega$, then $\J\preceq\fin^2\cap(\J\otimes\emptyset)$.
\end{lemma}

\begin{proof}
Denote $\I=\fin^2\cap(\J\otimes\emptyset)$. Fix any partition $(H_{i,j})\subseteq\J$ with $\emptyset\otimes\fin\subseteq\hat{\J}(H_{i,j})$. Since $\J$ is a weak P-ideal, there is a $C\notin\J$ with $C\cap H_{i,j}\in\fin$ for all $i,j\in\omega$. Let $f:\omega\to C$ be a finite-to-one function witnessing $\J|C\leq_{KB}\J$ (which exists, as $\J$ is $\leq_{KB}$-uniform). Define a partition $(G_{i,j})$ of $\omega^2$ by $G_{i,j}=f^{-1}[C\cap H_{i,j}]\times\omega$. Since $f$ is finite-to-one and $C\cap H_{i,j}\in\fin$, the set $f^{-1}[C\cap H_{i,j}]$ is finite, so $G_{i,j}\in\I$ for each $(i,j)\in\omega^2$. We claim that $\hat{\J}(H_{i,j})\cap\fin^2\subseteq\hat{\I}(G_{i,j})$. 

Let $A\in\hat{\J}(H_{i,j})\cap\fin^2$. Then each subset of $\bigcup_{(i,j)\in A}H_{i,j}$ with finite intersection with each $H_{i,j}$ is in $\J$. In particular, $C\cap\bigcup_{(i,j)\in A}H_{i,j}\in\J|C$ and consequently $f^{-1}[C\cap\bigcup_{(i,j)\in A} H_{i,j}]\in\J$. We need to show that $A\in\hat{\I}(G_{i,j})$. Fix any $X\subseteq\bigcup_{(i,j)\in A}G_{i,j}$ with $X\cap G_{i,j}$ finite for all $i,j\in\omega$. Then $X\in\emptyset\otimes\fin\subseteq\fin^2$ and
$$X\subseteq\bigcup_{(i,j)\in A}G_{i,j}=f^{-1}\left[C\cap\bigcup_{(i,j)\in A} H_{i,j}\right]\times\omega\in\J\otimes\emptyset.$$
Hence, $X\in\I$ and $A\in\hat{\I}(G_{i,j})$.
\end{proof}

\begin{example}
\label{ex_prec}
Even for unboring $\bf{\Sigma^0_4}$ ideals, $\leq_K$ and $\prec$ are not the same, i.e., there are unboring $\bf{\Sigma^0_4}$ ideals $\I$ and $\J$ such that $\J\not\leq_K\I$ but $\J\preceq\I$.

Indeed, Let $\J$ be a tall $\leq_{KB}$-uniform analytic P-ideal (for instance, $\J=\I_d$ -- see Example \ref{I_d is KB}). Then $\J$ is $\bf{\Pi^0_3}$ (see \cite[Lemma 1.2.2 and Theorem 1.2.5]{Farah}), so it is a hereditary weak P-ideal (by Proposition \ref{all-cor}). Moreover, $\I=\fin^2\cap(\J\otimes\emptyset)$ is $\bf{\Sigma^0_4}$ (as an intersection of two $\bf{\Sigma^0_4}$ ideals). Thus, by Lemmas \ref{l1}, \ref{l2} and \ref{l3}, we obtain the thesis.
\end{example}

\section{Mr\'{o}wka spaces revisited}

In this section we turn our attention to two particular cases: pairs of ideals $(\I,\J)$ such that $\I$ is extendable to a hereditary weak P-ideal $\I'$ such that $\J\not\leq_K\I'$ and pairs of ideals $(\I,\J)$ satisfying $\J\not\leq_K\I|A$ for all $A\notin\I$. Probably it is possible to state the results of this section in a more general way, however it would involve a lot of technicalities. 

\begin{lemma}
\label{new1}
If $\J$ is an ideal such that $\J\not\leq_K\J'$ for some ideal $\J'$, then $\J$ is tall.
\end{lemma}

\begin{proof}
If $\J$ would not be tall then any bijection between $\bigcup\J'$ and the set $B\notin\J$ with $\J|B=\fin(B)$ would witness $\J\leq_K\J'$.
\end{proof}

\begin{lemma}
\label{new2}
Assume that there is an AD family $\mathcal{A}\subseteq\J$ such that for every function $f\in\mathcal{D}_{\I}$ there is a $B\notin\I$ such that $f|B$ is finite-to-one and $f[B]\subseteq A$ for some $A\in\cA$. Then $\Phi(\cA)\in\text{FinBW}(\I)\setminus\text{FinBW}(\J)$. 
\end{lemma}

\begin{proof}
The fact that $\Phi(\cA)\in\text{FinBW}(\I)$ follows from Proposition \ref{MrowkaChar}. We need to show that $\Phi(\cA)\notin\text{FinBW}(\J)$. Without loss of generality we can assume that $\J$ is an ideal on $\omega$. We claim that the sequence $(n)_{n\in\omega}$ does not have a convergent subsequence indexed by some $B\notin\J$. Indeed, no subsequence of $(n)_{n\in\omega}$ converges to an element of $\omega$ and if $(n)_{n\in B}$ converges to some $A\in\cA$ then $B\setminus A\in\fin$, so $B\in\J$ (as $\mathcal{A}\subseteq\J$). Finally, no $(n)_{n\in B}$ converges to $\infty$ as $\Phi(\cA)$ is Hausdorff and $\cA$ is a MAD family, by Lemma \ref{MAD} ($B\cap A$ has to be infinite for some $A\in\cA$, so the sequence $(n)_{n\in B}$ cannot converge to $\infty$ as it has to have a subsequence $(n)_{n\in B\cap A}$ converging to $A\in\cA$).
\end{proof}

\begin{theorem}
\label{Dthm2}
Under CH, if $\I$ is unboring and $\J\not\leq_K\I|A$ for all $A\notin\I$, then there is a Mr\'{o}wka space in $\text{FinBW}(\I)\setminus\text{FinBW}(\J)$. Moreover, if $\I$ is strongly unboring, then CH can be relaxed to (MA($\sigma$-centered).
\end{theorem}

\begin{proof}
This proof is very similar to the proof of Theorem \ref{Mrowka}, so we will omit some details.

Enumerate $\cD_\I=\{f_\alpha:\ \omega\leq\alpha<\cc\}$. By induction on $\alpha$, we will construct a family $\mathcal{A}=\{A_\alpha:\alpha<\cc\}$ as in Lemma \ref{new2}. To start, find a partition $(A_n)$ of $\omega$ into sets belonging to $\J$ (which exists by Lemma \ref{new1}).

Assume that $\omega\leq\alpha<\cc$ and that $A_\beta\in\J$, for all $\beta<\alpha$, have been chosen. 

Suppose first that CH holds. If there is a $\beta<\alpha$ with $f_\alpha[C]\subseteq A_\beta$ for some $C\notin\I$ with $f_\alpha|C$ finite-to-one, then put $A_\alpha=A_\beta$. Otherwise, enumerate $\alpha$ as $(\beta_i)_{i\in\omega}$ and consider the partition $(B_i)$ of $\omega$, where $B_0=A_{\beta_0}$ and $B_i=A_{\beta_i}\setminus\bigcup_{j<i}A_{\beta_j}$ for all $0<i<\omega$. Since $\BI\not\sqsubseteq\I$, there is a $C\notin\I$ such that $f_\alpha|C\in\cD_{\fin(C)}$ and $f_\alpha[C]\cap B_i\in\fin$ for all $i\in\omega$. Since $\J\not\leq_K\I|C$, there is a $B\in\J$ with $f_\alpha^{-1}[B]\cap C\notin\I|C$. Then $f_\alpha|B$ is finite-to-one. Put $A_\alpha=f_\alpha[f_\alpha^{-1}[B]\cap C]\subseteq B\in\J$ and observe that $A_\alpha$ has finite intersection with each $A_\beta$ for $\beta<\alpha$ (as $A_\alpha\subseteq f_\alpha[C]$).

Suppose now that $\I$ is strongly unboring and MA($\sigma$-centered) holds. Then we can find $C\notin\I$ such that $f_\alpha|C$ is finite-to-one and $\I|C$ is $\omega$-diagonalizable by $(\mathcal{I}|C)^\star$-universal sets. If $C\cap f^{-1}[A_\beta]\notin\I$ for some $\beta<\alpha$ then put $A_\alpha=A_\beta$. Otherwise, perform the same construction as in the proof of Theorem \ref{MAlem} to get a set $B\subseteq C$, $B\notin\I|C$ such that $f_\alpha|B\in\cD_{\fin(B)}$ and $f_\alpha[B]\cap A_\beta$ is finite for all $\beta<\alpha$. 

Since $B\notin\I$, $\J\not\leq_K\I|B$. Thus, there is a $D\in\J$ such that $f_\alpha^{-1}[D]\cap B\notin\I|B$. Define $A_\alpha=f_\alpha[f_\alpha^{-1}[D]\cap B]\subseteq D\in\J$ and observe that $A_\alpha$ has finite intersection with each $A_\beta$ for $\beta<\alpha$ (as $A_\alpha\subseteq f_\alpha[B]$).
\end{proof}

\begin{theorem}
\label{Dthm}
Under CH, if $\I$ is extendable to a hereditary weak P-ideal $\I'$ such that $\J\not\leq_K\I'$, then there is a Mr\'{o}wka space in $\text{FinBW}(\I)\setminus\text{FinBW}(\J)$. Moreover, if $\I'$ is coanalytic, then CH can be relaxed to MA($\sigma$-centered).
\end{theorem}

\begin{proof}
Note that $\text{FinBW}(\I')\subseteq\text{FinBW}(\I)$ (as $\I\subseteq\I'$), so it suffices to find a Mr\'{o}wka space in $\text{FinBW}(\I')\setminus\text{FinBW}(\J)$. 

Let $\{f_\alpha:\ \omega\leq\alpha<\cc\}$ be an enumeration of $\mathcal{D}_{\I'}$. By induction on $\alpha$, we will construct a family $\mathcal{A}=\{A_\alpha:\alpha<\cc\}$ as in Lemma \ref{new2}. To start, find a partition $(A_n)$ of $\omega$ into sets belonging to $\J$ (which exists by Lemma \ref{new1}).

Assume that $\alpha<\cc$ and that $A_\beta\in\J$, for all $\beta<\alpha$, have been chosen. Since $\J\not\leq_K\I'$, there is a $H\in\J$ such that $f^{-1}_\alpha[H]\notin\I'$. 

Suppose first that CH holds. If there is a $\beta<\alpha$ with $f_\alpha[C]\subseteq A_\beta$ for some $C\notin\I'|f^{-1}_\alpha[H]$ with $f_\alpha|C$ finite-to-one, then put $A_\alpha=A_\beta$. Otherwise, enumerate $\alpha$ as $(\beta_i)_{i\in\omega}$. Consider the same partition $(B_i)$ as in the proof of Theorem \ref{Dthm2}. 

Since $\BI\not\sqsubseteq\I'|f^{-1}_\alpha[H]$ (by Proposition \ref{hh}), there is a $C\notin\I'|f^{-1}_\alpha[H]$ such that $f_\alpha|C\in\cD_{\fin(C)}$ and $f_\alpha[C]\cap B_i$ is finite for all $i\in\omega$ (see the proof of Theorem \ref{Mrowka} for details). Moreover, since $C\subseteq f^{-1}_\alpha[H]$, $f_\alpha[C]$ is in $\J$. Put $A_\alpha=f_\alpha[C]$ and observe that $A_\alpha\cap A_{\beta_i}$ is finite for each $i\in\omega$.

Suppose now that $\I'$ is coanalytic and MA($\sigma$-centered) holds. Since $\I'$ is a hereditary weak P-ideal, $\fin^2\not\sqsubseteq\I'|f_\alpha^{-1}[H]$ and we can find $C\notin\I'$, $C\subseteq f_\alpha^{-1}[H]$ such that $f_\alpha|C$ is finite-to-one (and $f_\alpha[C]\subseteq H\in\J$). 

If there is a $\beta<\alpha$ such that $C\cap f^{-1}[A_\beta]\notin\I'$ then put $A_\alpha=A_\beta$. If this is not the case, note that $\I'$ being a coanalytic hereditary weak P-ideal implies that $\I'$ is strongly unboring (by Proposition \ref{r}(b)). Hence, in the same way as in Theorem \ref{MAlem}, we can produce $B\subseteq C$ such that $B\notin\I'|C$ and $f_\alpha[B]\cap A_\beta$ is finite for all $\beta<\alpha$. Define $A_\alpha=f_\alpha[B]$ and note that $A_\alpha=f_\alpha[B]\subseteq f_\alpha[C]\subseteq H\in\J$. 
\end{proof}

\section{Distinguishing ideals}

\begin{proposition}
\label{prec_2}
If $\J$ is a weak P-ideal and $\J\preceq\I$ then $\text{FinBW}(\I)\subseteq\text{FinBW}(\J)$.
\end{proposition}

\begin{proof}
We can assume that $\I$ and $\J$ are ideals on $\omega$. Let $X\in\text{FinBW}(\I)$ and fix $g:\omega\to X$. We need to find $A\notin\J$ such that $(g(n))_{n\in A}$ converges. Without loss of generality we may assume that $g^{-1}[\{x\}]\in\J$ for all $x\in X$.

Suppose first that $\overline{g[\omega]}$ is a boring space. By Proposition \ref{boringspaces2}, since $\J$ is a weak P-ideal, there is an $A\notin\J$ with $(g(n))_{n\in A}$ convergent.

Assume now that $\overline{g[\omega]}$ is not boring and find $\{x_i:\ i\in\omega\}\subseteq X$, $x_i\neq x_j$ for all $i\neq j$, $i,j\in\omega$, and an infinite partition $(B_i)_{i\in\omega}$ of $g[\omega]$ into infinite sets such that for each $i\in\omega$ each injective sequence in $B_i$ converges to $x_i$. This can be done inductively. Indeed, let $\{z_i:\ i\in\omega\}$ be an enumeration of $g[\omega]$ and start the induction by finding any infinite $B'_0\subseteq g[\omega]$ which is convergent to some $x_0\in X$ ($B'_0$ exists as $X$ is sequentially compact) and putting $B_0=B'_0\cup\{z_0\}$. Assume now that $B_j$ and $x_j$ for all $j<i$ are already defined. Then for each $j,j'<i$, $j\neq j'$ there is an open set $V_{j,j'}$ such that $x_j\in V_{j,j'}$ and $x_{j'}\notin V_{j,j'}$. Define $U_j=\bigcap_{j'<i,j'\neq j}V_{j,j'}\ni x_j$. Note that since $\overline{g[\omega]}$ is not boring, $\overline{g[\omega]}\setminus\bigcup_{j<i}B_i$ cannot be boring. Thus, either $\left(\overline{g[\omega]}\setminus\bigcup_{j<i}B_i\right)\setminus\bigcup_{j<i}U_i$ is infinite and, using sequential compactness of $X$, we can find an infinite $B'_i\subseteq \left(\overline{g[\omega]}\setminus\bigcup_{j<i}B_i\right)\setminus\bigcup_{j<i}U_i$ which is convergent to some $x_i\in X$, or $\left(\overline{g[\omega]}\setminus\bigcup_{j<i}B_i\right)\setminus\bigcup_{j<i}U_i$ is finite, but there is a $j<i$ such that not every injective sequence in $U_j\cap \left(\overline{g[\omega]}\setminus\bigcup_{j<i}B_i\right)$ converges to $x_j$, so we can find an infinite $B'_i\subseteq U_j\cap \left(\overline{g[\omega]}\setminus\bigcup_{j<i}B_i\right)$ converging to some $x_i\notin\{x_j:\ j<i\}$. In both cases it suffices to put $B_i=B'_i\cup\{z_i\}$. 

Define $H_{i,j}=g^{-1}[\{b^i_j\}]\in\J$, where $B_i=\{b^i_j:\ j\in\omega\}$, for all $i,j\in\omega$. 

If $\fin\otimes\emptyset\not\subseteq\hat{\J}(H_{i,j})$ then there is an $i\in\omega$ such that $\{i\}\times\omega\notin\hat{\J}(H_{i,j})$. Thus, we can find $A\notin\J$ with $A\subseteq\bigcup_{j\in\omega}H_{i,j}$ and $A\cap H_{i,j}$ finite for all $j\in\omega$. In this case $g|A$ is finite-to-one (as $A\cap H_{i,j}$ is finite for all $j\in\omega$) and $(g(n))_{n\in A}$ converges to $x_i$ (as $g[A]\subseteq B_i$).

On the other hand, if $\fin\otimes\emptyset\subseteq\hat{\J}(H_{i,j})$, then by $\J\preceq\I$ we can find a partition $(G_{i,j})\subseteq\I$ such that $\hat{\J}(H_{i,j})\cap\fin^2\subseteq\hat{\I}(G_{i,j})$. Define $h:\omega\to X$ by 
$$h(n)=b^i_j\ \Longleftrightarrow\ n\in G_{i,j}$$
for all $n\in\omega$. Then $h^{-1}[\{x\}]\in\I$ for all $x\in X$. 

Since $X\in\text{FinBW}(\I)$, there is a $C\notin\I$ such that $(h(n))_{n\in C}$ converges in $X$. Since $X$ is Hausdorff, without loss of generality we may assume that $h|C$ is finite-to-one (there may be at most one $y\in X$ with $C\cap h^{-1}[\{y\}]$ infinite, so using the fact that $h^{-1}[\{y\}]\in\I$  it suffices to consider $C\setminus h^{-1}[\{y\}]\notin\I$ instead of $C$). Thus, $C\cap G_{i,j}\in\fin$ for all $i,j\in\omega$ (otherwise $h|C$ would not be finite-to-one) and 
$$D=\{(i,j)\in\omega^2:\ C\cap G_{i,j}\neq\emptyset\}\in\fin^2\setminus\hat{\I}(G_{i,j})\subseteq\fin^2\setminus\hat{\J}(H_{i,j}).$$ 
Indeed, $D\in\fin^2$ as $|D\cap(\{i\}\times\omega)|=\omega$ for at most one $i\in\omega$ ($|D\cap(\{i\}\times\omega)|=\omega=|D\cap(\{i'\}\times\omega)|$ for two distinct $i,i'\in\omega$ would imply that $(h(n))_{n\in C}$ has two subsequences -- one converging to $x_i$ and the other converging to $x_{i'}$). Moreover, $D\notin \hat{\I}(G_{i,j})$ as $C\subseteq\bigcup_{(i,j)\in D}G_{i,j}$ and $C\cap G_{i,j}\in\fin$ for all $i,j\in\omega$, but $C\notin\I$. Finally, $D\notin \hat{\J}(H_{i,j})$ as otherwise $D$ would belong to $\hat{\J}(H_{i,j})\cap\fin^2\subseteq\hat{\I}(G_{i,j})$.

Since $D\notin\hat{\J}(H_{i,j})$, there is an $A\notin\J$ with $\{(i,j)\in\omega^2:\ A\cap H_{i,j}\neq\emptyset\}\subseteq D$ and $A\cap H_{i,j}\in\fin$ for all $i,j\in\omega$. This implies that $g|A$ is finite-to-one and $(g(n))_{n\in A}$ is a subsequence of $(h(n))_{n\in C}$, hence, converges to the same limit as $(h(n))_{n\in C}$.
\end{proof}

\begin{corollary}\
\label{corollary}
If $\J\leq_K\I$ then $\text{FinBW}(\I)\subseteq\text{FinBW}(\J)$.
\end{corollary}

\begin{proof}
Note that if $\fin^2\sqsubseteq\J$ then $\fin^2\sqsubseteq\I$ (by Proposition \ref{propertyKat}) and $\text{FinBW}(\I)=\text{FinBW}(\J)$ by Theorem \ref{extreme}. On the other hand, if $\fin^2\not\sqsubseteq\J$ then we can apply Propositions \ref{p1}(b) and \ref{prec_2}.
\end{proof}

Our next result shows that the question whether FinBW($\I$) and FinBW($\J$) can be distinguished is easy, whenever at least one of the ideal $\I$ and $\J$ is boring.

\begin{corollary}
\label{IJ boring}
Let $\J$ be boring.
\begin{itemize}
\item[(a)] Suppose that $\fin^2\sqsubseteq\J$.
\begin{itemize}
\item[(i)] If $\fin^2\sqsubseteq\I$ then $\text{FinBW}(\I)=\text{FinBW}(\J)$.
\item[(ii)] If $\fin^2\not\sqsubseteq\I$ then $\text{FinBW}(\J)\subsetneq\text{FinBW}(\I)$. 
\end{itemize}
\item[(b)] Suppose that $\fin^2\not\sqsubseteq\J$.
\begin{itemize}
\item[(i)] If $\fin^2\sqsubseteq\I$ then $\text{FinBW}(\I)\subsetneq\text{FinBW}(\J)$.
\item[(ii)] If $\fin^2\not\sqsubseteq\I$ and $\I$ is boring then $\text{FinBW}(\I)=\text{FinBW}(\J)$.
\item[(iii)] If $\I$ is unboring then $\text{FinBW}(\J)\subsetneq\text{FinBW}(\I)$. 
\end{itemize}
\end{itemize}
\end{corollary}

\begin{proof}
Follows from Theorems \ref{extreme} and \ref{char}.
\end{proof}

\begin{theorem}
\label{char2}
(CH) The following are equivalent for any unboring ideal $\J$ and any hereditary weak P-ideal $\I$ (in particular, for any $\bf{\Pi^0_4}$ ideal $\I$):
\begin{itemize}
\item[(a)] $\text{FinBW}(\I)\setminus\text{FinBW}(\J)\neq\emptyset$;
\item[(b)] there is a Mr\'{o}wka space in $\text{FinBW}(\I)\setminus\text{FinBW}(\J)$;
\item[(c)] $\J\not\leq_K\I$;
\item[(d)] $\J\not\preceq\I$.
\end{itemize}
Moreover, if $\I$ is coanalytic, then it suffices to assume MA($\sigma$-centered) instead of CH.
\end{theorem}

\begin{proof}
(c)$\Longleftrightarrow$(d): This is Proposition \ref{1<->2}. 

(a)$\implies$(d): This is Proposition \ref{prec_2}. 

(c)$\implies$(b): This is Theorem \ref{Dthm}.

(b)$\implies$(a): Obvious.
\end{proof}

Since the implication (a)$\implies$(c) in Theorem \ref{char2} is true in ZFC, we are able to state the following result.

\begin{corollary}
Let $\I$ and $\J$ be $\bf{\Pi^0_4}$ ideals.
\begin{itemize}
\item[(a)] The following are equivalent:
\begin{itemize}
\item[(i)] $\J\not\leq_K\I$ in some forcing extension;
\item[(ii)] $\J\not\leq_K\I$ in every forcing extension;
\item[(iii)] $\text{FinBW}(\I)\setminus\text{FinBW}(\J)\neq\emptyset$ in every forcing extension in which MA($\sigma$-centered) holds;
\item[(iv)] there is a forcing extension in which $\text{FinBW}(\I)\setminus\text{FinBW}(\J)\neq\emptyset$.
\end{itemize}
\item[(b)] If either $\J\leq_K\I$ holds in ZFC or $\J\not\leq_K\I$ holds in ZFC, then the following are equivalent:
\begin{itemize}
\item[(i)] $\J\not\leq_K\I$;
\item[(ii)] under MA($\sigma$-centered) there is an uncountable separable space in FinBW($\I$);
\item[(iii)] it is consistent that there is an uncountable separable space in FinBW($\I$).
\end{itemize}
\end{itemize}
\end{corollary}

\begin{proof}
The proof is entirely similar to the proof of Corollary \ref{absolute} -- it suffices to replace Theorem \ref{char} with Theorem \ref{char2} and Corollary \ref{c} with Propositions \ref{1<->2} and \ref{prec_2}.
\end{proof}

\begin{example}
\label{y}
There are unboring $\bf{\Sigma^0_4}$ ideals $\I$ and $\J$ such that $\J\not\leq_K\I$ but $\text{FinBW}(\I)\subseteq\text{FinBW}(\J)$. Indeed, this follows from Example \ref{ex_prec} and Proposition \ref{prec_2}.
\end{example}

Theorem \ref{char2} characterizes $\text{FinBW}(\I)\setminus\text{FinBW}(\J)\neq\emptyset$ only in the case when $\I$ is a hereditary weak P-ideal. Thus, a natural question is whether it is possible that $\text{FinBW}(\I)$ and $\text{FinBW}(\widetilde{\I})$ are not the same, where $\widetilde{\I}$ is the smallest hereditary weak P-ideal containing $\I$ (which exists by Proposition \ref{smallest}). The answer is yes (even for $\bf{\Sigma^0_4}$ ideals). To show it we will need the following lemma.

\begin{lemma}
\label{lemconv}
For each $B\notin\conv$ there is a $C\subseteq B$, $C\notin\conv$ with $\fin^2\cong\conv|C$. 
\end{lemma}

\begin{proof}
As $B\notin\conv$, there is an injective sequence $(x_n)\subseteq\overline{B}$ such that each $x_n$ is a limit of some injective sequence in $B$. Since $[0,1]$ is sequentially compact, we can assume that $(x_n)$ converges in $[0,1]$ to some $x\in[0,1]$. Let $(U_n)$ be a decreasing sequence of open neighborhoods of $x$ such that $\lim_{n}\text{diam}(U_n)=0$ and $x_n\in U_n$ for each $n\in \omega$. For every $n$ find $B_n\subseteq B$ such that any injective sequence in $B_n$ converges to $x_n$. Define $C_0=U_0\cap B_0$ and $C_{n+1}=U_{n+1}\cap B_{n+1}\setminus \bigcup_{i\leq n}B_i$ for all $n$. Then $\conv|\bigcup_{n}C_n$ is isomorphic to $\fin^2$. 
\end{proof}

\begin{example}
Consider the $\bf{\Sigma^0_4}$ ideal $\conv$. It is extendable to the ideal $\text{null}=\{A\subseteq\mathbb{Q}\cap[0,1]:\ \overline{A}\text{ is of Lebesgue measure }0\}$. By Proposition \ref{all-cor}, $\text{null}$ is a hereditary weak P-ideal, as it is $\bf{\Pi^0_3}$ (see \cite{FarahSolecki}). Thus, by Proposition \ref{smallest}, there is a smallest hereditary weak P-ideal $\widetilde{\conv}$ containing $\conv$.

We claim that $\text{FinBW}(\conv)\setminus\text{FinBW}(\widetilde{\conv})\neq\emptyset$ under CH. In order to show it, we will prove that $\widetilde{\conv}\not\leq_K\conv|B$ for all $B\notin\conv$. As $\conv$ is unboring by Proposition \ref{conv}, the thesis will follow from Theorem \ref{Dthm2}.

Fix $B\notin\conv$. By Lemma \ref{lemconv}, there is a $C\notin\conv$, $C\subseteq B$ with $\conv|C$ isomorphic to $\fin^2$. Thus, it suffices to show $\widetilde{\conv}\not\leq_K\fin^2$. It will follow that $\widetilde{\conv}\not\leq_K\conv|C$ and consequently $\widetilde{\conv}\not\leq_K\conv|B$.

Fix $f:\omega^2\to\mathbb{Q}\cap[0,1]$. Since $[0,1]$ is sequentially compact, for each $n$ we can find $x_n\in[0,1]$ and $A_n\in[\omega]^\omega$ such that either $f[\{n\}\times A_n]=\{x_n\}$ or $f[\{n\}\times A_n]$ is infinite and each injective sequence in $f[\{n\}\times A_n]$ converges to $x_n$. There are also $x\in[0,1]$ and $A\in[\omega]^\omega$ such that $(x_n)_{n\in A}$ converges to $x$. Let $(U_n)_{n\in A}$ be a decreasing sequence of open neighborhoods of $x$ such that $\lim_{n}\text{diam}(U_n)=0$ and $x_n\in U_n$ for each $n\in A$. 

Define $D=\bigcup_{n\in A}\{n\}\times (A_n\cap f^{-1}[U_n])$. Then $\fin^2\not\ni D\subseteq f^{-1}[f[D]]$. However, $f[D]\in\widetilde{\conv}$. Indeed, either $f[D]\in\conv\subseteq\widetilde{\conv}$ or $\conv|f[D]$ contains an isomorphic copy of $\fin^2$. In the latter case, $f[D]\notin\widetilde{\conv}$ would mean that $\fin^2\sqsubseteq\conv|f[D]\subseteq\widetilde{\conv}|f[D]$ and contradict the fact that $\widetilde{\conv}$ is a hereditary weak P-ideal.
\end{example}

We end this section with an open problem.

\begin{problem}
Is it true that under CH for all Borel unboring ideals $\I$ and $\J$ the following are equivalent:
\begin{itemize}
\item[(i)] $\J\not\preceq\I$;
\item[(ii)] $\text{FinBW}(\I)\setminus\text{FinBW}(\J)\neq\emptyset$.
\end{itemize}
\end{problem}

Theorem \ref{char2} is a special case -- it shows that this is true for hereditary weak P-ideals. By Proposition \ref{prec_2}, the implication (ii)$\implies$(i) is true in general (even in ZFC). However, we were not able to show under CH that $\J\not\preceq\I$ implies existence of a Mr\'{o}wka space in $\text{FinBW}(\I)\setminus\text{FinBW}(\J)$.

\section{Applications}

\subsection{Simple density ideals}

Recall the definition of the density zero ideal: $\I_{d}=\left\{A\subseteq\omega:\ \lim_{n\to\infty}\frac{|A\cap\{0,1,\ldots,n\}|}{n+1}=0\right\}$. In this subsection we want to apply our results to summarize some facts about FinBW$(\I_d)$.

We need to introduce two classes of ideals which have been considered in the literature. 
\begin{itemize}
\item If $g:\omega\to [0,\infty)$ is such that $\lim_{n}g(n)=\infty$ and the sequence $\left(\frac{n}{g(n)}\right)$ does not converge to $0$, then we define the simple density ideal $\cZ_g$ by:
$$A\in\cZ_g\ \Longleftrightarrow\ \lim_{n}\frac{|A\cap\{0,1,\ldots,n\}|}{g(n)}=0.$$
This class of ideals was first studied in \cite{simpledensity} (see also \cite{EUvsSD} and \cite{Jarek}).
\item If $h:\omega\to [0,\infty)$ is such that $\sum_{i=0}^\infty h(i)=\infty$ and $\lim_{n}\frac{h(n)}{\sum_{i=0}^n h(i)}=0$, then we define an Erd\H{o}s-Ulam ideal $\EU_h$ by: 
$$A\in\EU_h\ \Longleftrightarrow\ \lim_{n}\frac{\sum_{i\in A\cap\{0,1,\ldots,n\}} h(i)}{\sum_{i=0}^{n} h(i)}=0$$
(cf. \cite[Example 1.2.3(d)]{Farah}).
\end{itemize}
Both Erd\H{o}s-Ulam ideals and simple density ideals are examples of analytic P-ideals (see \cite{simpledensity} and \cite[Example 1.2.3(d)]{Farah}). Clearly, the ideal $\I_d$ is both Erd\H{o}s-Ulam and simple density. It is known that there are $\cc$ pairwise non-isomorphic simple density ideals which are not Erd\H{o}s-Ulam as well as there are $\cc$ pairwise non-isomorphic Erd\H{o}s-Ulam ideals which are not simple density (see \cite[Proposition 5]{EUvsSD} and \cite[Theorem 3]{Jarek}). 

\begin{proposition}\
\begin{itemize}
\item[(a)] $\omega_1$ with order topology is in FinBW$(\I_d)$. Moreover, under MA($\sigma$-centered) there is an uncountable separable space in FinBW$(\I_d)$. 
\item[(b)] FinBW$(\EU_h)=$FinBW($\I_d$) for all Erd\H{o}s-Ulam ideals $\EU_h$. 
\item[(c)] $\text{FinBW}(\I_d)\subseteq\text{FinBW}(\cZ_g)$ for all simple density ideals $\cZ_g$. 
\item[(d)] $[0,1]\in\text{FinBW}(\cZ_g)\setminus\text{FinBW}(\I_d)$ for each simple density ideal $\cZ_{g}$ which is not Erd\H{o}s-Ulam. 
\item[(e)] (MA($\sigma$-centered)) There is a family $\mathcal{F}$ of size $\cc$ of simple density ideals such that $\text{FinBW}(\cZ_{g})\setminus\text{FinBW}(\cZ_{g'})\neq\emptyset$ whenever $\cZ_{g},\cZ_{g'}\in\cF$ are distinct.
\end{itemize}
\end{proposition}

\begin{proof}
(a): Follows from Corollary \ref{corollary-main}.

(b): By \cite[Theorem 1.13.10]{Farah}, all Erd\H{o}s-Ulam ideals are $\leq_K$-equivalent, i.e., $\EU_h\leq_K\EU_{h'}$ whenever $\EU_h$ and $\EU_{h'}$ are Erd\H{o}s-Ulam ideals (in fact, all Erd\H{o}s-Ulam ideals are even $\leq_{RB}$-equivalent). Thus, it suffices to apply Corollary \ref{corollary}.

(c): Follows from Corollary \ref{corollary}, as $\cZ_g\leq_{KB}\I_d$ for each simple density ideal $\cZ_g$ (by \cite[Theorem 4 and Remark 2]{Jarek}).

(d): By \cite[Section 3]{Gdansk} $[0,1]\notin\text{FinBW}(\I_d)$. On the other hand, each simple density ideal $\cZ_{g}$ which is not Erd\H{o}s-Ulam is extendable to a $\bf{\Sigma^0_2}$ ideal (hence, $[0,1]\in\text{FinBW}(\cZ_g)$ by \cite[Theorem 4.2]{Gdansk}). Indeed, each tall density ideal in the sense of Farah, which is not an Erd\H{o}s-Ulam ideal (i.e., belonging to the class $(\mathcal{Z}4)$ from \cite[Lemma 1.13.9]{Farah}), is generated by some unbounded submeasure $\phi$, hence, it can be extended to a $\bf{\Sigma^0_2}$ ideal $\fin(\phi)$. Thus, it suffices to note that all simple density ideals are tall density ideals in the sense of Farah (by \cite[Theorem 3.2]{simpledensity}).

(e): By \cite[Lemma 1.2.2 and Theorem 1.2.5]{Farah}, each analytic P-ideal is $\bf{\Pi^0_3}$. Thus, by Proposition \ref{all-cor}, Erd\H{o}s-Ulam ideals and simple density ideals are hereditary weak P-ideals. What is more, by \cite[Theorem 3]{Jarek}, there is a family of cardinality $\cc$ consisting of simple density ideals pairwise incomparable with respect to $\leq_K$. Hence, the thesis follows from Theorem \ref{char2}.
\end{proof}

\subsection{Hindman spaces}
\label{sectionHindman}

\begin{definition}
For an infinite set $M\subseteq\omega$ we write $FS(M)=\{\sum_{n\in F}n:\ F\in[M]^{<\omega},F\neq\emptyset\}$. A topological space is Hindman if for each sequence $(x_n)\subseteq X$ we can find $x\in X$ and an infinite $D\subseteq\omega$ such that for each open neighborhood $U$ of $x$ there is an $m\in\omega$ with $\{x_n:\ n\in FS(D\setminus \{0,1,\ldots,m\})\}\subseteq U$.
\end{definition}

Hindman spaces were introduced by M. Kojman in \cite{KojmanHindman}. Since then this subject has been extensively studied among others in \cite{FT}, \cite{Flaskova}, \cite{Jones}, \cite{KojmanShelah} and recently in \cite{NewHindman}, where the following result is proved:

\begin{proposition}[{\cite[Proposition 1.1]{NewHindman}}]
\label{H1}
Every Mr\'{o}wka space defined by a MAD family is not a Hindman space.
\end{proposition}

\begin{lemma}
\label{H2}
Every boring space is Hindman.
\end{lemma}

\begin{proof}
By \cite[Theorem 11]{KojmanHindman}, a space $X$ is Hindman whenever it satisfies condition $(\star)$, that is, whenever $X$ is Hausdorff and the closure in $X$ of each countable subset of $X$ is compact and first-countable. Using Proposition \ref{boringspaces}(ii), it is easy to see that each boring space satisfies condition $(\star)$.
\end{proof}

\begin{corollary}
\label{2}
(CH) The following are equivalent for every ideal $\I$:
\begin{itemize}
\item $\I$ is unboring;
\item there is an uncountable Mr\'{o}wka space in FinBW($\I$) which is not a Hindman space;
\item there is a space in FinBW($\I$) which is not a Hindman space.
\end{itemize}
\end{corollary}

\begin{proof}
Follows from Theorem \ref{char}, Proposition \ref{H1} (as any uncountable Mr\'{o}wka space given by Theorem \ref{char} is defined by a MAD family) and Lemma \ref{H2}.
\end{proof}

The following result generalizes the main theorem of \cite{KojmanShelah} stating that under CH Hindman spaces do not coincide with the class FinBW($\cW$) (see also \cite{Jones} where this is proved under MA($\sigma$-centered) instead of CH).

\begin{corollary}
(CH) There is no ideal $\I$ such that FinBW($\I$) coincides with Hindman spaces.
\end{corollary}

\begin{proof}
If $\I$ is a boring ideal then $[0,1]$ is a Hindman space (cf. \cite{KojmanHindman}) not belonging to FinBW($\I$) (by Theorem \ref{extreme}). On the other hand, if $\I$ is unboring then it suffices to apply Corollary \ref{2}.
\end{proof}

\subsection*{Acknowledgment}
The Author is grateful to R. Filip\'{o}w for remarks that allowed to improve Proposition \ref{boringspaces} and to the anonymous referee for his helpful comments that refined the quality of the manuscript.


\end{document}